\documentclass[2pt]{article}
\usepackage{amssymb,amsmath,amsfonts,amsthm,enumerate,color}
\usepackage[toc,page,title,titletoc,header]{appendix}
\usepackage{graphicx}
\usepackage{indentfirst}
\usepackage{multicol}
\usepackage{booktabs}
\usepackage{url}
 \usepackage{algorithm}
\usepackage{algorithmic}
\usepackage{graphicx}
\usepackage{subfigure}
\usepackage{graphics}

\numberwithin{equation}{section}

\newtheorem{theorem}{Theorem}

\newtheorem{definition}{Definition}

\newtheorem{remark}{Remark}
\newtheorem{lemma}{Lemma}
\newtheorem{proposition}{Proposition}

\def \Vh0{\stackrel{\circ}{V}_h}

\newcommand{\lc}
{\mathrel{\raise2pt\hbox{${\mathop<\limits_{\raise1pt\hbox
{\mbox{$\sim$}}}}$}}}

\newcommand{\gc}
{\mathrel{\raise2pt\hbox{${\mathop>\limits_{\raise1pt\hbox{\mbox{$\sim$}}}}$}}}

\newcommand{\ec}
{\mathrel{\raise2pt\hbox{${\mathop=\limits_{\raise1pt\hbox{\mbox{$\sim$}}}}$}}}

\def\bb{\begin{equation}} \def\ee{\end{equation}}

\def\beqn{\begin{eqnarray}}  \def\eqn{\end{eqnarray}}

\def\beqnx{\begin{eqnarray*}} \def\eqnx{\end{eqnarray*}}

\def\bn{\begin{enumerate}} \def\en{\end{enumerate}}

\def\bd{\begin{description}} \def\ed{\end{description}}



\begin{document}

\title{Iteratively Linearized Reweighted Alternating Direction Method of Multipliers for a Class of Nonconvex Problems}
\author{Tao Sun\thanks{
Department of Mathematics, National University of Defense Technology,
Changsha, 410073, Hunan,  China. Email: \texttt{nudtsuntao@163.com} } \and Hao Jiang\thanks{College of Computer, National University of Defense Technology,
Changsha, 410073, Hunan,  China. Email: \texttt{haojiang@nudt.edu.cn}}
\and Lizhi Cheng$^{*}$\thanks{The State Key Laboratory for High Performance Computation, National University of Defense Technology,
Changsha, 410073, Hunan,  China. Email: \texttt{clzcheng@nudt.edu.cn}}
\and Wei Zhu\thanks{Hunan Key Laboratory for Computation and Simulation in Science and Engineering, School of Mathematics and Computational Science, Xiangtan University, Xiangtan, Hunan, 411105, China.  Email: \texttt{zhuwei@xtu.edu.cn}}}

%

\maketitle

\begin{abstract}
In this paper, we consider solving a class of nonconvex and nonsmooth problems frequently appearing in signal processing and machine learning research. The traditional alternating direction method of multipliers encounters troubles in both mathematics and computations  in solving the nonconvex and nonsmooth subproblem. In view of this, we propose a reweighted alternating direction method of multipliers. In this algorithm,  all subproblems are  convex and easy to  solve. We also provide several guarantees for the  convergence and prove that  the algorithm globally converges to a critical point of an auxiliary function with the help of the Kurdyka-{\L}ojasiewicz property. Several numerical results are presented to demonstrate the efficiency of the proposed algorithm.
\end{abstract}

\section{Introduction}
Minimization of composite functions with  linear constrains finds various applications in signal and image processing, statistics, machine learning, to name a few.
Mathematically, such a problem can be presented as
\begin{equation}\label{omodel}
    \min_{x,y}\{f(x)+g(y) ~~\textrm{s.t.}~~Ax+By=c\},
\end{equation}
where $A\in \mathbb{R}^{r\times M}$, $B\in \mathbb{R}^{r\times N}$, and $g$ is usually the regularization function, and $f$ is usually the loss function.

The well-known alternating direction method of multipliers (ADMM) method \cite{gabay1976dual,glowinski1975approximation} is a powerful tool for the problem mentioned above. The ADMM actually focuses on the augmented Lagrangian problem of (\ref{omodel}) which reads as
\begin{align}
&\mathcal{\widetilde{L}}_{\alpha}(x,y,p):=f(x)+g(y)+\langle p,Ax+By-c\rangle+\frac{\alpha}{2}\|Ax+By-c\|_2^2,
\end{align}
where $\alpha>0$ is a parameter. The ADMM   minimizes only one variable and fixes others in each iteration; the variable $p$ is updated by a feedback strategy. Mathematically, the standard ADMM method can be presented as
 \begin{eqnarray}\label{oscheme}
 \left\{\begin{array}{lcl}
          y^{k+1}&=&\textrm{arg}\min_{y} \mathcal{\widetilde{L}}_{\alpha}(x^k,y,p^k) \\
          x^{k+1}&=&\textrm{arg}\min_{x} \mathcal{\widetilde{L}}_{\alpha}(x,y^{k+1},p^k) \\
          p^{k+1}&=&p^k+\alpha(Ax^{k+1}+By^{k+1}-c)
        \end{array}
 \right.
\end{eqnarray}

The ADMM algorithm attracts increasing attention for its  efficiency in dealing with sparsity-related problems \cite{yang2010fast,wen2010alternating,xu2012alternating,ng2010solving,wang2008new}. Obviously, the ADMM has a self-explanatory assumption; all the subproblems shall be solved efficiently. In fact, if the proximal maps  of the $f$ and $g$ are easy to  calculate, the  linearized  ADMM \cite{chen1994proximal} proposes the linearized technique to solve the subproblem efficiently; the subproblems all need to compute  proximal map of $f$ or $g$ once. The core part of the linearized ADMM lies in linearizing the quadratic terms $\frac{\alpha}{2}\|Ax+By^k-c\|_2^2$ and $\frac{\alpha}{2}\|Ax^{k+1}+By-c\|_2^2$ in each iteration. The linearized ADMM is also called as preconditioned ADMM in  \cite{esser2010general}; in fact, it is also a special case when $\theta=1$ in Chambolle-Pock primal dual algorithm \cite{chambolle2011first}. In the latter paper \cite{wang2014bregman}, the linearized ADMM is further generalized as the Bregman ADMM.

The convergence of the ADMM in the convex case is also well studied; numerous excellent works have made contributions to this field \cite{he20121,he2015non,hong2017linear,deng2016global}.  Recently, the ADMM algorithm is even developed for the infeasible problems \cite{liu2017new,ryu2018douglas}.  The earlier analyses  focus on the convex case, i.e., both $f$ and $g$ are all convex. But as the nonconvex penalty functions perform efficiently in applications, nonconvex ADMM is developed and studied: in paper \cite{chartrand2013nonconvex}, Chartrand and Brendt  directly used the ADMM to the group sparsity problems.  They replace the nonconvex subproblems as a class of proximal maps. Later, Ames and Hong consider applying ADMM for certain non-convex quadratic problems \cite{ames2016alternating}. The convergence is also presented. A class of nonconvex problems is solved by Hong et al by a provably convergent ADMM \cite{hong2016convergence}.  They also allow the subproblems to be solved inexactly by taking gradient steps which can be regarded as a linearization. Recently, with weaker assumptions, \cite{wang2015global} present new analysis for nonconvex ADMM by novel mathematical techniques. With the Kurdyka-{\L}ojasiewicz property, \cite{li2015global,li2016douglas} consider the  convergence of the generated iterative points. \cite{sun2017alternating} considered a structured constrained problem and proposed the ADMM-DC algorithm. In nonconvex ADMM literature, either the proximal maps of $f$ and $g$ or the subproblems  are assumed to be easily solved.
\subsection{Motivating example and problem formulation}
This subsection contains two parts: the first one presents an example and discusses the problems in direct using the  ADMM; the second one describes the problem considered in this paper.
\subsubsection{A motivating example: the problems in directly using ADMM}
The methods mentioned above are feasibly applicable provided the subproblems are relatively easy to solve, i.e., either the proximal maps of $f$ and $g$ or the subproblems  are assumed to be easily solved. However, the nonconvex cases may not always promise  such a  convention. We recall the  $\textrm{TV}^q_{\varepsilon}$ problem \cite{hintermuler2013nonconvex} which arises in imaging science
\begin{equation}\label{tvq}
    \min_{u}\{\frac{1}{2}\|f-\Psi u\|_2^2+\sigma\|Tu\|_{q,\varepsilon}^q\},
\end{equation}
where $T$ is the total variation operator and $\|v\|_{q,\varepsilon}^q:=\sum_{i}(|v_i|+\varepsilon)^q$. By denoting $v=Tu$, the problem then turns to being
\begin{equation}\label{tvqc}
    \min_{u,v}\{\frac{1}{2}\|f-\Psi u\|_2^2+\sigma\|v\|_{q,\varepsilon}^q,~~~\textrm{s.t.}~~Tu-v=\textbf{0}\}.
\end{equation}
 The direct ADMM for this problem can be presented as
  \begin{eqnarray}\label{tvqadmm}
  \left\{\begin{array}{lcl}
           v^{k+1}&=&\textrm{arg}\min_{v}\{\sigma \|v\|_{q,\varepsilon}^q+\langle p^k,v\rangle+\frac{\alpha}{2}\|v-Tu^k\|_2^2\}, \\
           u^{k+1}&=&\textrm{arg}\min_{u}\{ \frac{1}{2}\|f-\Psi u\|_2^2-\langle p^k,Tu\rangle+\frac{\alpha}{2}\|v^{k+1}-Tu\|_2^2\}, \\
           p^{k+1}&=&p^k+\alpha(v^{k+1}-Tu^{k+1}).
         \end{array}
  \right.
\end{eqnarray}
The first subproblem in the algorithm needs to minimize a nonconvex and nonsmooth problem. If $q=\frac{1}{2},\frac{2}{3}$, the point $v^k$ can be explicitly calculated. This is because the proximal map of $\|\cdot\|_{q,\varepsilon}^q$ can be easily obtained. But for other $q$, the proximal map cannot be easily derived. Thus, we may must employ  iterative algorithms to compute $v^{k+1}$. That indicates three drawbacks which cannot be ignored:

\begin{enumerate}
    \item The stopping criterion is hard to set for the nonconvexity\footnote{The convex methods usually enjoy a convergence rate.}.
    \item The error may be accumulating in the iterations due to the inexact numerical solution of the subproblem.
    \item Even the subproblem can be numerically solved without any error, the numerical solution for the subproblem is always  a critical  point rather than the ``real"  \textrm{argmin} due to the nonconvexity.
\end{enumerate}

In fact, the other penalty functions like Logistic function \cite{weston2003use}, Exponential-Type Penalty (ETP) \cite{gao2011feasible}, Geman \cite{geman1995nonlinear}, Laplace \cite{trzasko2009highly} also encounter such a problem.

\subsubsection{Optimization problem and basic assumptions}
In this paper, we consider the following problem
\begin{equation}\label{model}
    \min_{x,y} f(x)+\sum_{i=1}^{N}g[h(y_i)] ~~\textrm{s.t.}~~Ax+By=c,
\end{equation}
where $A\in \mathbb{R}^{r\times M}$,  and $f$, $g$ and $h$ satisfy the following assumptions:

\textbf{A.1} $f:\mathbb{R}^N\rightarrow \mathbb{R}$ is a   differentiable convex function with a Lipschitz continuous gradient, i.e.,
\begin{equation}
    \| \nabla f(x)-\nabla f(y)\|_2\leq L_f\| x-y\|_2, \forall x,y\in \mathbb{R}^N.
\end{equation}
And the function $f(x)+\frac{\|Ax\|_2^2}{2}$ is strongly convex with constant $\delta$.

\textbf{A.2} $h:\mathbb{R}\rightarrow \mathbb{R}$ is  convex and proximable.

\textbf{A.3} $g:\textrm{Im}(h)\rightarrow \mathbb{R}$ is a differentiable concave function with a Lipschitz continuous gradient whose Lipschitz
continuity modulus is bounded by $L_g>0$; that is
\begin{equation}
    \mid g'(s)-g'(t)\mid\leq L_g\mid s-t\mid,
\end{equation}
and $g'(t)>0$ when $t\in \textrm{Im}(h)$.

It is easy to see that the $\textrm{TV}^q$ problem can be regarded as a special one of (\ref{model}) if we set $g(s)=(s+\varepsilon)^{q}$ and $h(t)=|t|$. The augmented lagrange dual function of model (\ref{model}) is
\begin{align}\label{dual}
    &\mathcal{L}_{\alpha}(x,y,p)=f(x)+\sum_{i=1}^{N}g[h(y_i)]+\langle p,Ax+By-c\rangle+\frac{\alpha}{2}\|Ax+By-c\|_2^2,
\end{align}
where $\alpha>0$ is a parameter.
\subsection{Linearized ADMM  meets the iteratively reweighted strategy: convexifying the subproblems}
In this part, we present the algorithm for solving problem (\ref{model}). The term $\sum_{i=1}^{N}g[h(y_i)]$ has a deep relationship with several iteratively reweighted style algorithms \cite{chartrand2008iteratively,daubechies2010iteratively,sun2017global,zhang2010analysis,lu2014proximal}. Although the function $\sum_{i=1}^{N}g[h(y_i)]$ may be nondifferentiable itself, the reweighted style methods still propose an elegant way: linearization of outside function $g$. Precisely, in $(k+1)$-th iteration of the iteratively reweighted style algorithms, the term $\sum_{i=1}^{N}g[h(y_i)]$ is usually replaced by $\sum_{i=1}^{N}g'[h(y_i^k)]\cdot[h(y_i)-h(y_i^k)]+\sum_{i=1}^{N}g[h(y_i^k)]$, where $y^k$ is obtained in the $k$-th iteration. The extensions of reweighted style methods  to matrix cases are considered and analyzed in \cite{lu2015generalized,lu2017ell,lu2015smoothed,sun2017convergence,lu2016nonconvex}. In fact, the  iteratively reweighted technique is a special majorization minimization technique, which  has also been adopted in ADMM \cite{lu2017unified}. Compared with \cite{lu2017unified}, the most difference in our paper is the exploiting the specific  structure of the problem in nonconvex settings.  Motivated by the iteratively reweighted strategy, we propose the following scheme for solving (\ref{model})
\begin{eqnarray}\label{scheme}
\left\{
\begin{array}{lcl}
  y^{k+1}&=&\textrm{arg}\min_{y} \{\sum_{i=1}^{N}g'[h(y^{k}_i)]h(y_i)+\langle p^k+\alpha(Ax^k+By^k-c), By\rangle+\frac{r}{2}\|y-y^{k+1}\|_2^2\}, \\
  x^{k+1}&=&\textrm{arg}\min_{x}\{f(x)+\langle p^k, Ax\rangle+\frac{\alpha}{2}\|Ax+By^{k+1}-c\|_2^2\}, \\
  p^{k+1}&=&p^k+\alpha(Ax^{k+1}+y^{k+1}-c).
\end{array}
\right.
\end{eqnarray}
We combined both linearized ADMM and reweighted algorithm in the new scheme: for the  nonconvex part $\sum_{i=1}^{N}g[h(y_i)]$, we linearize the outside function $g$ and keep $h$, which aims to derive the convexity of the subproblem; for the quadratic part $\frac{\alpha}{2}\|Ax^{k+1}+By-c\|_2^2$,  linearization is for the use of the proximal map of $h$. We call this new algorithm as  Iteratively Linearized Reweighted Alternating Direction Method of Multipliers (ILR-ADMM). It is easy to see that each subproblem just needs to solve a convex problem in this scheme. With the expression of proximal maps, updating $y^{k+1}$ can be equivalently presented as the following forms
\small
 \begin{eqnarray}\label{schemeprox}
  y^{k+1}_i=\textbf{prox}_{\frac{g'[h(y^{k}_i)]}{r}h}(y^{k}_i-\frac{B^{\top}_i(\alpha(Ax^k+By^{k}-c)+p^k)}{r}),
\end{eqnarray}
\normalsize
where $i\in [1,2,\ldots,N]$, and $B_i$ denotes the $i$-th column of the matrix $B$. In many applications, $f$ is the quadratic function, and then solving $x^{k+1}$ is also very easy. With this form, the algorithm can be programmed with the absence of the inner loop.
\begin{algorithm}
\caption{Iteratively Linearized Reweighted Alternating Direction Method of Multipliers (ILR-ADMM)}
\begin{algorithmic}\label{alg1}
\REQUIRE   parameters $\alpha>0,r>0$\\
\textbf{Initialization}: $x^0,y^{0},p^0$\\
\textbf{for}~$k=0,1,2,\ldots$ \\
~~~ $y^{k+1}_i=\textbf{prox}_{\frac{g'[h(y^{k}_i)]}{r}h}(y^{k}_i-\frac{B^{\top}_i(\alpha(Ax^k+By^{k}-c)+p^k)}{r}), i\in [1,2,\ldots,N]$, \\
~~~$x^{k+1}=\textrm{arg}\min_{x}\{f(x)+\langle p^k, Ax\rangle+\frac{\alpha}{2}\|Ax+By^{k+1}-c\|_2^2\}$, \\
~~~$p^{k+1}=p^k+\alpha(Ax^{k+1}+By^{k+1}-c)$ \\
\textbf{end for}\\
\end{algorithmic}
\end{algorithm}
\subsection{Contribution and Organization}
In this paper, we consider  a class of nonconvex and nonsmooth problems which are ubiquitous in applications. Direct use of ADMM algorithms will   lead to troubles in both computations  and mathematics for the nonconvexity of the subproblem. In view of this, we propose  the iteratively  linearized reweighted  alternating direction method of multipliers for these problems. The new algorithm is a combination of iteratively reweighted strategy and the linearized ADMM. All the subproblems in the proposed algorithm are  convex and easy to solve if the proximal map of $h$ is easy to solve and $f$ is quadratic. Compared with the direct application of  ADMM to problem (\ref{model}), we now list the advantages of the new algorithm:
\begin{enumerate}
    \item Computational perspective: each subproblem just needs to compute once proximal map of  $g$ and minimize a quadratic problem, the computational cost is low in each iteration.
    \item Practical perspective: without any inner loop, the programming is very easy.
    \item Mathematical perspective: all the subproblems is convex and exactly solved. Thus, we get ``real" \textrm{argmin} everywhere, which makes the mathematical convergence analysis solid and meaningful.
\end{enumerate}
With the help of the Kurdyka-{\L}ojasiewicz property, we provide the convergence results of the algorithm with proper selections of the parameters.  The applications of the new algorithm  to the signal and image processing are presented. The numerical results demonstrate the efficiency of the proposed algorithm.

The rest of this paper is organized as follows. Section 2 introduces the preliminaries including the definitions of subdifferential and the  Kurdyka-{\L}ojasiewicz property.  Section 3   provides the convergence analysis. The core part is using an auxiliary  Lyapunov function and  bounding the generated sequence. Section 4 applies the proposed algorithm to image deblurring. And several comparisons are reported. Finally, Section 5 concludes the paper.
\section{Preliminaries}
We introduce the basic tools in the analysis: the subdifferential and Kurdyka-{\L}ojasiewicz property. These two definitions play important roles in the variational analysis.
\subsection{Subdifferential}
Given a lower semicontinuous function $J: \mathbb{R}^N\rightarrow (-\infty,+\infty]$, its domain is defined by
$$\textrm{dom} (J):=\{x\in \mathbb{R}^N: J(x)<+\infty\}.$$
The graph of a real extended valued function $J: \mathbb{R}^N \rightarrow (-\infty, +\infty]$ is defined by
$$\textrm{graph} (J):=\{(x,v)\in\mathbb{R}^N\times \mathbb{R}: v=J(x)\}.$$
Now, we are prepared to present the definition  of subdifferential. More details can be found in  \cite{rockafellar2009variational}.
\begin{definition}Let  $J: \mathbb{R}^N \rightarrow (-\infty, +\infty]$ be a proper and lower semicontinuous function.
\begin{enumerate}
  \item For a given $x\in \textrm{dom} (J)$, the Fr$\acute{e}$chet subdifferential of $J$ at $x$, written as $\hat{\partial}J (x)$, is the set of all vectors $u\in \mathbb{R}^N$ satisfying
  $$\lim_{y\neq x}\inf_{y\rightarrow x}\frac{J(y)-J(x)-\langle u, y-x\rangle}{\|y-x\|_2}\geq 0.$$
When $x\notin \textrm{dom} (J)$, we set $\hat{\partial}J(x)=\emptyset$.

\item The (limiting) subdifferential, or simply the subdifferential, of $J$ at $x\in  \textrm{dom} (J)$, written as $\partial J(x)$, is defined through the following closure process
\begin{align*}
\partial J(x):=\{u\in\mathbb{R}^N: \exists x^k\rightarrow x, J(x^k)\rightarrow J(x)~\textrm{and}~ u^k\in \hat{\partial}J(x^k)\rightarrow u~\textrm{as}~k\rightarrow \infty\}
\end{align*}
\end{enumerate}
\end{definition}
Note that if $x\notin \textrm{dom} (J)$,  $\partial J(x)=\emptyset$. When $J$ is convex,  the definition agrees with the  subgradient in convex analysis \cite{rockafellar2015convex} which is defined as
$$\partial J(x):=\{v\in \mathbb{R}^N: J(y)\geq J(x)+\langle v,y-x\rangle~~\textrm{for}~~\textrm{any}~~y\in \mathbb{R}^N\}.$$
It is easy to verify that the Fr$\acute{e}$chet subdifferential is convex and closed while the subdifferential is closed. Denote that
$$\textrm{graph} (\partial J):=\{(x,v)\in\mathbb{R}^N\times \mathbb{R}^N: v\in \partial J(x)\},$$
thus, $\textrm{graph} (\partial J)$ is a closed set.
Let $\{(x^k,v^k)\}_{k\in \mathbb{N}}$ be a sequence in $\mathbb{R}^N\times \mathbb{R}$ such that $(x^k,v^k)\in \textrm{graph }(\partial J)$. If $(x^k,v^k)$  converges to $(x, v)$ as $k\rightarrow +\infty$ and $J(x^k)$ converges to $v$ as $k\rightarrow +\infty$, then $(x, v)\in \textrm{graph }(\partial J)$. This indicates the following simple proposition.
\begin{proposition}\label{sublimit}
If $\{x^k\}_{k=0,1,2,\ldots}\subseteq \textrm{dom}(J)$, $v^k\in \partial J(x^k)$, $\lim_{k}v^k=v$, $\lim_{k}x^k=x\in \textrm{dom} (J)$, and $\lim_{k} J(x^k)=J(x)$\footnote{If $J$ is continuous, the condition $\lim_{k} J(x^k)=J(x)$ certainly holds if $\lim_{k}x^k=x\in \textrm{dom} (J)$.}. Then, we have
\begin{equation}
    v\in \partial J(x).
\end{equation}
\end{proposition}
A necessary condition for $x\in\mathbb{R}^N$ to be a minimizer of $J(x)$ is
\begin{equation}\label{Fermat}
\textbf{0}\in \partial J(x).
\end{equation}
When $J$ is convex, (\ref{Fermat}) is also sufficient.
\begin{definition}
A point that satisfies (\ref{Fermat}) is called (limiting) critical point. The set of critical points of $J(x)$ is denoted by $\textrm{crit}(J)$.
\end{definition}

\begin{proposition}\label{criL}
If $(x^*,y^*,p^*)$ is a critical point of $\mathcal{L}_{\alpha}(x,y,p)$ with any $\alpha>0$, it must hold that
\begin{eqnarray}
 -B^{\top}p^{*}&\in& W^{*}\partial h(y^{*}),\nonumber\\
-A^{\top}p^{*}&=& \nabla f(x^{*}),\nonumber\\
Ax^{*}+By^{*}-c&=&\textbf{0},\nonumber
\end{eqnarray}
where $\mathcal{L}_{\alpha}(x,y,p)$ is defined in (\ref{dual}) and $W^*=\textrm{Diag}\{g'[h(y^*_i)]\}_{1\leq i\leq N}$.
\end{proposition}
\begin{proof}
With [Proposition 10.5, \cite{rockafellar2009variational}], we have
\begin{equation}
     \partial(\sum_{i=1}^Ng[h(x_i)])=\partial (g[h(x_1)])\times\ldots\times\partial (g[h(x_N)])
\end{equation}
Noting that $g$ is differentiable and $h$ is convex, with direct computation, we have
\begin{equation}
    \hat{\partial} (g[h(x_i)])=g'[h(x_i)]\cdot \partial h(x_i).
\end{equation}
And more, by definition, we can obtain
\begin{equation}
    \partial (g[h(x_i)])=g'[h(x_i)]\cdot \partial h(x_i).
\end{equation}
We then prove the first equation. The second and third are quite easy.
\end{proof}
\subsection{Kurdyka-{\L}ojasiewicz function}
The domain of a subdifferential is given as
$$\textrm{dom} (\partial J):=\{x\in \mathbb{R}^N: \partial J(x)\neq \emptyset\}.$$
\begin{definition}\label{KLdefi}
(a) The function $J: \mathbb{R}^N \rightarrow (-\infty, +\infty]$ is said to have the  Kurdyka-{\L}ojasiewicz property at $\overline{x}\in \textrm{dom}(\partial J)$ if there
 exist $\eta\in (0, +\infty)$, a neighborhood $U$ of $\overline{x}$ and a continuous concave function $\varphi: [0, \eta)\rightarrow \mathbb{R}^+$ such that
\begin{enumerate}
  \item $\varphi(0)=0$.
  \item $\varphi$ is $C^1$ on $(0, \eta)$.
  \item for all $s\in(0, \eta)$, $\varphi^{'}(s)>0$.
  \item for all $x$ in $U\bigcap\{x|J(\overline{x})<J(x)<J(\overline{x})+\eta\}$, it holds
\begin{equation}
  \varphi^{'}(J(x)-J(\overline{x}))\cdot\textrm{dist}(\textbf{0},\partial J(x))\geq 1.
\end{equation}
\end{enumerate}

(b) Proper closed functions which satisfy the Kurdyka-{\L}ojasiewicz property at each point of $\textrm{dom}(\partial J)$ are called KL functions.
\end{definition}
More details can be  found in \cite{lojasiewicz1993geometrie,kurdyka1998gradients,bolte2007lojasiewicz}. In the following part of the paper, we use KL for Kurdyka-{\L}ojasiewicz  for short. Directly checking whether a function is KL or not is hard, but the proper closed semi-algebraic functions \cite{bolte2007lojasiewicz} do much help.
\begin{definition}
(a) A subset $S$ of $\mathbb{R}^N$ is a real semi-algebraic set if there exists a finite number of real polynomial functions $g_{ij}, h_{ij}:\mathbb{R}^N\rightarrow \mathbb{R}$ such that
$$S=\bigcup_{j=1}^p\bigcap_{i=1}^q\{u\in \mathbb{R}^N:g_{ij}(u)=0~\textrm{and}~~h_{ij}(u)<0\}.$$

(b) A function $h:\mathbb{R}^N\rightarrow (-\infty, +\infty]$ is called semi-algebraic if its graph
$$\{(u, t)\in \mathbb{R}^{N+1}: h(u)=t\}$$
is a semi-algebraic subset of $\mathbb{R}^{N+1}$.
\end{definition}
Better yet, the semi-algebraicity enjoys many quite nice properties and various kinds of functions are KL \cite{attouch2013convergence}. We just put a few of them here:
\begin{itemize}
\item Real closed polynomial functions.
\item Indicator functions of closed semi-algebraic sets.
\item Finite sums and product of closed semi-algebraic functions.
\item The composition of closed semi-algebraic functions.
\item Sup/Inf type function, e.g., $\sup\{g (u, v) : v \in C\}$ is semi-algebraic when $g$ is a closed semi-algebraic
function and $C$ a closed semi-algebraic set.
\item Closed-cone of PSD matrices, closed Stiefel manifolds and closed constant rank matrices.
\end{itemize}

\begin{lemma}[\cite{bolte2007lojasiewicz}]\label{semikl}
Let $J:\mathbb{R}^N\rightarrow \mathbb{R}$ be a proper and closed function. If $J$ is semi-algebraic then it satisfies the KL property
at any point of $\textrm{dom} (J)$.
\end{lemma}
The previous definition and property of KL is about a certain point in $\textrm{dom}(J)$. In fact, the property has been extended to a certain closed set \cite{bolte2014proximal}. And this property makes previous convergence proofs related to KL property much easier.
\begin{lemma}\label{con}
Let $J: \mathbb{R}^N\rightarrow \mathbb{R}$ be a proper lower semi-continuous function and $\Omega$ be a compact set. If $J$ is a constant on $\Omega$ and $J$ satisfies the KL property at each point on $\Omega$, then there exists concave function $\varphi$ satisfying the four properties given in Definition \ref{KLdefi} and $\eta,\varepsilon>0$ such that for any $\overline{x}\in \Omega$ and any $x$ satisfying that $\textrm{dist}(x,\Omega)<\varepsilon$ and $f(\overline{x})<f(x)<f(\overline{x})+\eta$, it holds that
\begin{equation}
    \varphi^{'}(J(x)-J(\overline{x}))\cdot\textrm{dist}(\textbf{0},\partial J(x))\geq 1.
\end{equation}
\end{lemma}
\section{Convergence analysis}
In this part, the function $\mathcal{L}_{\alpha}(x,y,p)$ is defined in (\ref{dual}). We provide the convergence guarantee and the convergence analysis of ILR-ADMM (Algorithm 1).
We first present a sketch of the proofs, which is also a big picture for the purpose of each lemma and theorem:
\begin{itemize}
\item In the first step, we bound the dual variables by the primal points (Lemma \ref{fp}).
\item In the second step, the sufficient descent condition is derived for a new Lyapunov function (Lemma \ref{descend}).
\item In the third step, we provide several conditions to bound the points (Lemma \ref{boundness}).
\item In the fourth step, the relative error condition is proved (Lemma \ref{relative}).
\item In the last step, we prove the convergence under semi-algebraic assumption (Theorem \ref{confinal}).
\end{itemize}
The proofs in our paper are closely related to seminal papers \cite{hong2016convergence,wang2015global,li2015global} in
several proofs treatments. In fact, some proofs follow their techniques. For example, in Lemma \ref{fp}, we employ the method used in [Lemma 3, \cite{wang2015global}] to
bound $\|p^{k+1}-p^k\|_2$. In Lemma \ref{boundness},  boundedness of  the sequence is also proved by a similar way given in [Theorem 3, \cite{li2015global}].   Besides the detailed issues, in the large picture, the keystones are
also similar to \cite{hong2016convergence,wang2015global,li2015global}: we also prove the sufficient descent and subdifferential bound for a
Lyapunov function, and the boundedness of the generated points.

However, the proofs in our paper are still different from  \cite{hong2016convergence,wang2015global,li2015global} in various
aspects. The novelties mainly lay in deriving the   sufficient descent and subdifferential bound based on the specific structure of our problem. Noting that in each iteration, we minimize  $\mathcal{L}_{\alpha}^k(x^k,y,p^k)$ and $\mathcal{L}_{\alpha}^k(x,y^{k+1},p^k)$ rather than
$\mathcal{L}_{\alpha}(x^k,y,p^k)$ and $\mathcal{L}_{\alpha}(x,y^{k+1},p^k)$. Thus, the previous methods cannot be directly used in our paper. By exploiting the structure property of the problem, we built these two conditions.

\begin{lemma}\label{fp}
If \begin{equation}\label{conditionp}
    \textrm{Im}(B)\bigcup \{c\}\subseteq \textrm{Im}(A).
\end{equation}
 Then, we have
\begin{equation}
    \|p^k-p^{k+1}\|_2^2\leq \eta\|x^{k+1}-x^{k}\|_2^2,
\end{equation}
where $\eta=\frac{L_f^2}{\theta^2}$, and $\theta$ is the smallest strictly-positive eigenvalue of $(A^{\top}A)^{1/2}$.
\end{lemma}
\begin{proof}
The second step in each iteration actually gives
\begin{align}
  \nabla f(x^{k+1})=-A^{\top}(\alpha(Ax^{k+1}+By^{k+1}-c)+p^k).
\end{align}
With the expression of $p^{k+1}$,
\begin{align}
 \nabla f(x^{k+1})=-A^{\top}p^{k+1}.
\end{align}
Replacing $k+1$ with $k$, we can obtain
\begin{align}\label{remarkused}
\nabla f(x^{k})=-A^{\top}p^{k}
\end{align}
Under condition (\ref{conditionp}), $p^{k+1}-p^k\in \textrm{Im}(A)$;  and  subtraction of the two equations above gives
\begin{align}
&\|p^k-p^{k+1}\|_2\leq \frac{1}{\theta}\|A^{\top}(p^k-p^{k+1})\|_2\nonumber\\
&\quad\leq\frac{\|\nabla f(x^{k+1})-\nabla f(x^{k})\|_2}{\theta}\leq\frac{L_f}{\theta}\|x^{k+1}-x^{k}\|_2.
\end{align}
\end{proof}
\begin{remark}
If condition (\ref{conditionp}) holds and $p^0\in \textrm{Im}(A)$, we have $p^k\in \textrm{Im}(A)$. Then, from (\ref{remarkused}), we have that
\begin{align}\label{boundp}
\|p^k\|_2\leq\frac{1}{\theta}\|\nabla f(x^k)\|_2.
\end{align}
We will use this inequality in bounding the sequence.
\end{remark}

\begin{remark}
The condition (\ref{conditionp}) is    satisfied  if $A$ is surjective. However, in many applications, the matrix $A$  may fail to be surjective. For example, for  a matrix $U\in \mathbb{R}^{N\times N}$, we consider the operator
\begin{equation}
    \mathcal{T}(U)=(DU,UD^{\top})\in \mathbb{R}^{(N-1)N}\times\mathbb{R}^{N(N-1)},
\end{equation}
where $D\in \mathbb{R}^{(N-1)\times N}$ is the forward difference operator. Noting the $\textrm{dim}(\textrm{Im}(\mathcal{T}))=2N(N-1)>N^2=\textrm{dim}(\textrm{dom}(\mathcal{T}))$ when $N>2$, thus, $\mathcal{T}$ cannot be  surjective in this case. However, the current convergence of nonconvex ADMM is all based on the surjective assumption on $A$ or condition (\ref{conditionp}), which is also used in our analysis.  How to remove condition (\ref{conditionp}) in the nonconvex ADMM deserves further research.
\end{remark}

 Now, we introduce several notation to  present the following lemma. Denote the variable  $d$ and the sequence $d^k$ as
\begin{equation}
    d:=(x,y,p), d^k:=(x^k,y^k,p^k), z^k:=(x^k,y^k).
\end{equation}
An auxiliary function is always used in the proof
 \begin{align}
    \mathcal{L}_{\alpha}^k(x,y,p)&:=f(x)+\sum_{i=1}^{N}g'[h(y_i^k)]h(y_i)+\langle p,Ax+By-c\rangle+\frac{\alpha}{2}\|Ax+By-c\|_2^2.
\end{align}

\begin{lemma}[Descent]\label{descend}
Let the sequence $\{(x^k,y^k,p^k)\}_{k=0,1,2,\ldots}$ be generated by ILR-ADMM. If condition (\ref{conditionp}) and the following condition
\begin{equation}\label{condition}
    \alpha>\max\{1,\frac{2\eta}{\delta}\},r>\alpha\|B\|_2^2
\end{equation}
hold, then there exists $\nu>0$ such that
\begin{equation}
    \mathcal{L}_{\alpha}(d^k)-\mathcal{L}_{\alpha}(d^{k+1})\geq\nu\|z^{k+1}-z^k\|_2^2,
\end{equation}
where $\mathcal{L}_{\alpha}(d^k)=\mathcal{L}_{\alpha}(x^k,y^k,p^k)$.
\end{lemma}
\begin{proof}
Direct calculation shows that the first step is actually minimizing the function $\mathcal{L}_{\alpha}^k(x^k,y,p^k)+\frac{(y-y^k)^{\top}(r_1 \mathbb{I}-\alpha B^{\top}B)(y-y^k)}{2}$ with respect to $y$. Thus, we have
\small
\begin{align}
&\mathcal{L}_{\alpha}^k(x^k,y^{k+1},p^k)+\frac{r-\alpha\|B\|_2^2}{2}\|y^{k+1}-y^k\|_2^2\nonumber\\
    &\quad\leq\mathcal{L}_{\alpha}^k(x^k,y^{k+1},p^k)+\frac{(y^{k+1}-y^k)^{\top}(r \mathbb{I}-\alpha B^{\top}B)(y^{k+1}-y^k)}{2}\leq \mathcal{L}_{\alpha}^k(x^k,y^{k},p^k).\nonumber
\end{align}
\normalsize
Similarly, $x^{k+1}$ actually minimizes $\mathcal{L}_{\alpha}^k(x^k,y,p^k)$. Noting $\alpha\geq 1$, with assumption \textbf{A.1}, the strongly convex constant of $\mathcal{L}_{\alpha}^k(x^k,y,p^k)$ is larger than $\delta$,
\small
\begin{align}
        \mathcal{L}_{\alpha}^k(x^{k+1},y^{k+1},p^k)&+\frac{\delta}{2}\|x^{k+1}-x^k\|_2^2\leq \mathcal{L}_{\alpha}^k(x^k,y^{k+1},p^k).
\end{align}
\normalsize
Direct calculation yields
\begin{align}
     &\mathcal{L}_{\alpha}^k(x^{k+1},y^{k+1},p^{k+1})=\mathcal{L}_{\alpha}^k(x^{k+1},y^{k+1},p^{k})+\langle p^{k+1}-p^k,Ax^{k+1}+By^{k+1}-c\rangle\nonumber\\
     &\quad\quad=\mathcal{L}_{\alpha}^k(x^{k+1},y^{k+1},p^{k})+\frac{1}{\alpha}\|p^{k+1}-p^k\|_2^2.
\end{align}
Combining the equations above, we can have
\small
\begin{align}\label{l1t1}
    &\mathcal{L}_{\alpha}^k(x^k,y^{k},p^k)\geq \mathcal{L}_{\alpha}^k(x^{k+1},y^{k+1},p^{k+1})+\frac{\delta}{2}\|x^{k+1}-x^k\|_2^2+\frac{r-\alpha\|B\|_2^2}{2}\|y^{k+1}-y^k\|_2^2-\frac{1}{\alpha}\|p^{k+1}-p^k\|_2^2.
\end{align}
\normalsize
Noting $g$ is concave, we have
\begin{align}\label{des1temp}
    &\sum_{i=1}^{N}g[h(y^k_i)]-\sum_{i=1}^{N}g[h(y^{k+1}_i)]\nonumber\\
    &\quad\quad=\sum_{i=1}^{N}\{g[h(y^k_i)]-g[h(y^{k+1}_i)]\}\nonumber\\
    &\quad\quad\geq\sum_{i=1}^{N}g'[h(y^k_i)][h(y^k_i)-h(y^{k+1}_i)]\\
    &\quad\quad=\sum_{i=1}^{N}g'[h(y^k_i)]h(y^k_i)-\sum_{i=1}^{N}g'[h(y^k_i)]h(y^{k+1}_i).\nonumber
\end{align}
Then, we can derive
\begin{align}
    &\mathcal{L}_{\alpha}(x^k,y^k,p^k)-\mathcal{L}_{\alpha}(x^{k+1},y^{k+1},p^{k+1})\nonumber\\
    &\quad\quad=\sum_{i=1}^{N}g[h(y^k_i)]-\sum_{i=1}^{N}g[h(y^{k+1}_i)]+f(x^k)+\langle p^k,Ax^k+By^k-c\rangle+\frac{\alpha}{2}\|Ax^k+By^k-c\|_2^2\nonumber\\
     &\quad\quad-\{f(x^{k+1})+\langle p^{k+1},Ax^{k+1}+By^{k+1}-c\rangle+\frac{\alpha}{2}\|Ax^{k+1}+By^{k+1}-c\|_2^2\}\nonumber\\
    &\quad\quad\geq\sum_{i=1}^{N}g'[h(y^k_i)]h(y^k_i)-\sum_{i=1}^{N}g'[h(y^k_i)]h(y^{k+1}_i)+f(x^k)+\langle p^k,Ax^k+By^k-c\rangle+\frac{\alpha}{2}\|Ax^k+By^k-c\|_2^2\nonumber\\
    &\quad\quad-\{f(x^{k+1})+\langle p^{k+1},Ax^{k+1}+By^{k+1}-c\rangle+\frac{\alpha}{2}\|Ax^{k+1}+By^{k+1}-c\|_2^2\}\nonumber\\
    &\quad\quad=\mathcal{L}_{\alpha}^k(x^k,y^{k},p^k)-\mathcal{L}_{\alpha}^k(x^{k+1},y^{k+1},p^{k+1})\nonumber\\
    &\quad\quad\geq \frac{\delta}{2}\|x^{k+1}-x^k\|_2^2+\frac{r-\alpha\|B\|_2^2}{2}\|y^{k+1}-y^k\|_2^2-\frac{1}{\alpha}\|p^{k+1}-p^k\|_2^2.\nonumber
\end{align}
With Lemma \ref{fp}, we then have
\begin{align}
    &\mathcal{L}_{\alpha}(x^k,y^k,p^k)-\mathcal{L}_{\alpha}(x^{k+1},y^{k+1},p^{k+1})\nonumber\\
     &\quad\geq(\frac{\delta}{2}-\frac{\eta}{\alpha})\|x^{k+1}-x^k\|_2^2+\frac{r-\alpha\|B\|_2^2}{2}\|y^{k+1}-y^k\|_2^2.
\end{align}
Letting $\nu:=\min\{\frac{\delta}{2}-\frac{\eta}{\alpha},\frac{r-\alpha\|B\|_2^2}{2}\}$, we then prove the result.
\end{proof}

In fact,  condition (\ref{condition}) can be always satisfied in applications because the parameters $r$ and $\alpha$ are both selected by the user.  Different with the ADMMs in convex setting, the parameter $\alpha$ is nonarbitrary, the $\alpha$ here should be sufficiently large.

\begin{lemma}[Boundedness]\label{boundness}
If $p^0\in \textrm{Im}(A)$ and conditions (\ref{conditionp}) and (\ref{condition}) hold, and there exists $\sigma_0>0$ such that
\begin{align}\label{quad}
\inf\{f(x)-\sigma_0\|\nabla f(x)\|_2^2\}>-\infty,
\end{align}
and
\begin{align}\label{alphabound}
\alpha\geq\frac{1}{2\sigma_0\theta^2}.
\end{align}
The sequence $\{d^k\}_{k=0,1,2,\ldots}$ is bounded, if  one of the following conditions holds:

\textbf{B1}. $g(y)$ is coercive, and $f(x)-\sigma_0\|\nabla f(x)\|_2^2$ is coercive.

\textbf{B2}. $g(y)$ is coercive,  and $A$ is invertible.

\textbf{B3}. $\inf\{g(y)\}>-\infty$, $f(x)-\sigma_0\|\nabla f(x)\|_2^2$ is coercive, and $A$ is invertible.
\end{lemma}
\begin{proof}
We have
\begin{align}
&\mathcal{L}_{\alpha}(d^k)=f(x^k)+g(y^k)+\langle p^k,Ax^k+By^k-c\rangle+\frac{\alpha}{2}\|Ax^k+By^k-c\|_2^2\nonumber\\
&=f(x^k)+g(y^k)-\frac{\|p^k\|_2^2}{2\alpha}+\frac{\alpha}{2}\|Ax^k+By^k-c+\frac{p^k}{\alpha}\|_2^2\nonumber\\
&=f(x^k)+g(y^k)-\sigma_0\theta^2\|p^k\|_2^2+(\sigma_0\theta^2-\frac{1}{2\alpha})\|p^k\|_2^2+\frac{\alpha}{2}\|Ax^k+By^k-c+\frac{p^k}{\alpha}\|_2^2\nonumber\\
(\ref{boundp})&\geq f(x^k)-\sigma_0\|\nabla f(x^k)\|_2^2+g(y^k)+(\sigma_0\theta^2-\frac{1}{2\alpha})\|p^k\|_2^2+\frac{\alpha}{2}\|Ax^k+By^k-c+\frac{p^k}{\alpha}\|_2^2.
\end{align}
Noting $\{\mathcal{L}_{\alpha}(d^k)\}_{k=0,1,2,\ldots}$ is decreasing with Lemma \ref{descend}, $\sup_{k}\mathcal{L}_{\alpha}(d^k)=\mathcal{L}_{\alpha}(d^0)$. We then can see $\{g(y^k)\}_{k=0,1,2,\ldots}$, $\{p^k\}_{k=0,1,2,\ldots}$, $\{Ax^k+By^k-c+\frac{p^k}{\alpha}\}_{k=0,1,2,\ldots}$ are all bounded. It is easy to see that if one of the three conditions holds, $\{d^k\}_{k=0,1,2,\ldots}$ will be bounded.
\end{proof}

\begin{remark}
The condition (\ref{quad}) holds for many quadratic functions \cite{li2015global,sun2017precompact}. This condition also implies the function $f$ is similar to quadratic function and its property is ``good".
\end{remark}

\begin{remark}
 Both assumptions \textbf{B2} and \textbf{B3} actually imply condition (\ref{boundp}).
\end{remark}

\begin{remark}
Combining  conditions (\ref{alphabound})  and (\ref{condition}), we then have
\begin{equation}\label{conpara}
   \alpha>\max\{1,\frac{2\eta}{\delta},\frac{1}{2\sigma_0\theta^2}\},r>\alpha\|B\|_2^2.
\end{equation}
To determine the $\alpha$, we need to obtain $\sigma_0$, $\theta$ and $\delta$ first. Computing these constants may be hard due to that they may fail to enjoy the explicit forms. Thus, in the experiments, we use an increasing technique introduced in \cite{lu2017nonconvex}.
\end{remark}

\begin{lemma}[Relative error]\label{relative}
If conditions (\ref{conditionp}), (\ref{conpara}), and (\ref{quad}) hold, and one of assumptions \textbf{B1}, \textbf{B2} and \textbf{B3} holds. Then for any $k\in \mathbb{Z}_+$, there exists $\tau>0$ such that
\begin{equation}
    \textrm{dist}(\textbf{0},\partial\mathcal{L}_{\alpha}(d^{k+1}))\leq \tau\|z^k-z^{k+1}\|_2.
\end{equation}
\end{lemma}
\begin{proof}
Due to that $h$ is convex, $h$ is Lipschitz continuous with some constant if being restricted to some bounded set.
Thus, there exists $L_h>0$ such that
$$|h(y^{k+1}_i)-h(y^{k}_i)|\leq L_h|y^{k+1}_i-y^{k}_i|.$$
Updating $y^{k+1}$ in each iteration certainly yields
\begin{align}
    r(y^k-y^{k+1})&-B^{\top}(\alpha(Ax^k+By^{k}-c)+p^k)\in W^k\partial h(y^{k+1}),
\end{align}
where $h(y):=\sum_{i=1}^N h(y_i)$ and $W^{k}:=\textrm{Diag}\{g'[h(y^{k}_1)],g'[h(y^{k}_2)],\ldots,g'[h(y^{k}_N)]\}$. Noting the boundedness of the sequence and $h(y^{k}_i)$, the continuity of $g'$ indicates there exist $\delta_1,\delta_2>0$ such that
\begin{equation}
    \delta_1\leq g'[h(y^{k}_i)]\leq \delta_2, i\in[1,2,\ldots,N], k\in \mathbb{Z}_+.
\end{equation}
Easy computation gives
\small
\begin{align}\label{ypart}
    &W^{k+1}(W^k)^{-1}[ r(y^k-y^{k+1})-\alpha B^{\top}(Ax^k+By^{k}-c)\nonumber\\
    &\quad-B^{\top}p^k]+B^{\top}p^{k+1}+\alpha B^{\top}(Ax^{k+1}+By^{k+1}-c)\in \partial_{y}\mathcal{L}_{\alpha}(d^{k+1}).
\end{align}
\normalsize
With the boundedness of the generated points, there exist $R_1>0$ such that
\begin{equation}
    \|r(y^k-y^{k+1})-\alpha B^{\top}(Ax^k+By^{k}-c)-B^{\top}p^k\|_2\leq R_1.
\end{equation}
Thus, we have
\begin{align}
    &\textrm{dist}(\textbf{0},\partial_{y}\mathcal{L}_{\alpha}(d^{k+1}))\leq\frac{R_1}{\delta_1}\|W^{k+1}-W^{k}\|_2\nonumber\\
    &\quad+\|B^{\top}p^{k+1}-B^{\top}p^{k}\|_2+\|\alpha B^{\top}A(x^{k+1}-x^{k})\|_2\nonumber\\
    &\quad+\|\alpha B^{\top}B(y^{k+1}-y^{k})\|_2+r\|y^{k+1}-y^{k}\|_2\nonumber\\
    &\quad\leq\frac{R_1}{\delta_1}\|W^{k+1}-W^{k}\|_2+\|B\|_2\cdot\|p^{k+1}-p^{k}\|_2\nonumber\\
    &\quad+\alpha\|B\|_2\cdot\|A\|_2\cdot\|x^{k+1}-x^{k}\|_2\nonumber\\
    &\quad+(\alpha\|B\|_2^2+r)\cdot\|y^{k+1}-y^{k}\|_2.
\end{align}
Obviously, it holds that
\begin{align}
    &\|W^{k+1}-W^{k}\|_2\leq\max_{i}|g'[h(y^{k+1}_i)]-g'[h(y^{k}_i)]|\nonumber\\
    &\quad\leq L_g\max_{i}|h(y^{k+1}_i)-h(y^{k}_i)|\nonumber\\
    &\quad\leq L_g L_h\|y^{k+1}-y^{k}\|_{\infty}\leq L_g L_h\|y^{k+1}-y^{k}\|_{2}.\nonumber
\end{align}
Thus, with Lemma \ref{fp}, we derive that
\begin{equation}
     \textrm{dist}(\textbf{0},\partial_{y}\mathcal{L}_{\alpha}(d^{k+1}))\leq \tau_y\|z^{k+1}-z^k\|_2,
\end{equation}
for $\tau_y=\max\{\frac{L_g L_h R_1}{\delta_1}+\alpha\|B\|_2^2+r,\alpha\|B\|_2\|A\|_2+\|B\|_2\sqrt{\eta}\}$.
From the second step in each iteration,
\begin{align}\label{xpart1}
    \nabla f(x^{k+1})=-A^{\top}p^{k}-\alpha A^{\top}(Ax^{k+1}+By^{k+1}-c).
\end{align}
Direct calculation gives
 \begin{align}\label{xpart}
   & \nabla f(x^{k+1})+A^{\top}p^{k+1}+\alpha A^{\top}(Ax^{k+1}+By^{k+1}-c)\in \partial_{x} \mathcal{L}_{\alpha}(d^{k+1}).
\end{align}
That is also
 \begin{align}\label{xpart3}
 A^{\top}(p^{k+1}-p^k)\in \partial_{x} \mathcal{L}_{\alpha}(d^{k+1}).
\end{align}
With Lemma \ref{fp}, we have
\begin{align}
    &\textrm{dist}(\textbf{0},\partial_{x}\mathcal{L}_{\alpha}(d^{k+1}))\leq\|A^{\top}p^{k+1}-A^{\top}p^{k}\|_2\nonumber\\
    &\quad\leq\|A\|_2\cdot\|p^{k+1}-p^{k}\|_2\leq \|A\|_2\sqrt{\eta}\|x^{k+1}-x^{k}\|_2\leq\tau_x\|z^{k+1}-z^k\|_2,\nonumber
\end{align}
where $\tau_x=\|A\|_2\sqrt{\eta}$.
It is easy to see
\begin{equation}
    \frac{p^{k+1}-p^k}{\alpha}=Ax^{k+1}+By^{k+1}-c\in \partial_{p} \mathcal{L}_{\alpha}(d^{k+1}).
\end{equation}
And we have
\begin{align}
     &\textrm{dist}(\textbf{0},\partial_{p}\mathcal{L}_{\alpha}(d^{k+1}))\leq
      \tau_{p}\|x^{k+1}-x^k\|_2\leq  \tau_{p}\|z^{k+1}-z^k\|_2,
\end{align}
for   $\tau_p=\frac{\sqrt{\eta}}{\alpha}$.
With the deductions above,
\begin{align}
    \textrm{dist}(\textbf{0},\partial\mathcal{L}_{\alpha}(d^{k+1}))&\leq(\tau_x+\tau_y+\tau_p)\times(\|z^{k+1}-z^k\|_2).
\end{align}
Denoting $\tau:=\tau_x+\tau_y+\tau_p$, we then finish the proof.
\end{proof}

\begin{lemma}
If the sequence $\{(x^k,y^k,p^k)\}_{k=0,1,2,\ldots}$ is bounded and conditions of  Lemma \ref{relative} hold, then we have
\begin{equation}
    \lim_{k}\|z^{k+1}-z^k\|_2=0.
\end{equation}
For any cluster point $(x^*,y^*,p^*)$, it is also a critical point of $\mathcal{L}_{\alpha}(x,y,p)$.
\end{lemma}
\begin{proof}
We can easily see that $\{d^k\}_{k=0,1,2,\ldots}$ is also bounded. The continuity of $\mathcal{L}_{\alpha}$ indicates that $\{\mathcal{L}_{\alpha}(d^k)\}_{k=0,1,2,\ldots}$ is bounded. From Lemma \ref{descend},  $\mathcal{L}_{\alpha}(d^k)$ is decreasing. Thus, the sequence $\{\mathcal{L}_{\alpha}(d^k)\}_{k=0,1,2,\ldots}$ is convergent, i.e., $\lim_{k}[\mathcal{L}_{\alpha}(d^k)-\mathcal{L}_{\alpha}(d^{k+1})]=0$.  With Lemma \ref{descend}, we have
\begin{equation}
    \lim_{k}\|z^{k+1}-z^k\|_2\leq\lim_{k}\sqrt{\frac{\xi(d^k)-\xi(d^{k+1})}{\nu}}=0.
\end{equation}
From the scheme of the ILR-ADMM, we also have
\begin{equation}
    \lim_{k}\|p^{k+1}-p^k\|_2=0.
\end{equation}
For any cluster point $(x^*,y^*,p^*)$, there exists $\{k_j\}_{j=0,1,2,\ldots}$ such that $\lim_{j}(x^{k_j},y^{k_j},p^{k_j})=(x^*,y^*,z^*)$. Then, we further have $\lim_{j}z^{k_j+1}=(x^*,y^*)$. From Lemma \ref{fp}, we also have $\lim_{j}p^{k_j+1}=p^*$.  That also means
\begin{equation}
    \lim_{j}W^{k_j}=W^*.
\end{equation}
From the scheme, we have the following conditions
\begin{align}
 &(W^{k_j})^{-1}[r(y^{k_j}-y^{k_j+1})-B^{\top}(\alpha(Ax^{k_j}+By^{k_j}-c)+p^{k_j})]\in \partial h(y^{k_j+1}),\nonumber\\
&-A^{\top}p^{k_j}-\alpha A^{\top}(Ax^{k_j+1}+By^{k_j+1}-c)=\nabla f(x^{k_j+1}),\nonumber\\
&\quad\quad p^{k_j+1}=p^{k_j}+\alpha(Ax^{k_j+1}+By^{k_j+1}-c).\nonumber
\end{align}
Letting $j\rightarrow+\infty$, with Proposition \ref{sublimit}, we have
\begin{eqnarray}
 (W^{*})^{-1}[-B^{\top}p^{*}]&\in& \partial h(y^{*}),\nonumber\\
-A^{\top}p^{*}&=& \nabla f(x^{*}),\nonumber\\
Ax^{*}+By^{*}-c&=&\textbf{0}.\nonumber
\end{eqnarray}
The first relation above is actually $-B^{\top}p^{*}\in W^{*}\partial h(y^{*})$. From Proposition \ref{criL}, $(x^*,y^*,z^*)$ is a critical point of $\mathcal{L}_{\alpha}$.
\end{proof}

In the following, to prove the convergence result,  we first establish some results about the limit points of the sequence generated by ILR-ADMM. These results are presented for the use of Lemma \ref{con}.
We recall a  definition of the  limit point which is introduced in \cite{bolte2014proximal}.
\begin{definition}
Define that
\begin{align}
    &\mathcal{M}(d^0):=\{d\in \mathbb{R}^N: \exists~\textrm{an increasing sequence of integers}~~\{k_j\}_{j\in\mathbb{N}}~\textrm{such that} ~d^{k_j}\rightarrow d ~\textrm{as}~ j\rightarrow \infty\},\nonumber
\end{align}
 where $d^0\in \mathbb{R}^n$ is an arbitrary starting point.
\end{definition}

\begin{lemma}\label{station}
Suppose that  the conditions of Lemma \ref{relative} hold, and  $\{d^{k}\}_{k=0,1,2,\ldots}$ is generated by scheme (\ref{scheme}). Then, we have the following results.

(1) $\mathcal{M}(d^0)$ is nonempty and $\mathcal{M}(d^0)\subseteq \textrm{cri}(\mathcal{L}_{\alpha})$.

(2) $\lim_{k}\textrm{dist}(d^k,\mathcal{M}(d^0))=0$.

(3) $\mathcal{L}_{\alpha}$ is finite and constant on $\mathcal{M}(d^0)$.
\end{lemma}
\begin{proof}
(1) Due to that $\{d^{k}\}_{k=0,1,2,\ldots}$  is bounded, $\mathcal{M}(d^0)$ is nonempty. Assume that $d^*\in \mathcal{M}(d^0)$, from the definition, there exists
a subsequence $d^{k_i}\rightarrow d^*$. From Lemma \ref{descend}, we have $d^{k_i-1}\rightarrow d^*$. From Lemma \ref{relative}, there exists $\omega^{k_i}\in \partial \mathcal{L}_{\alpha}(d^{k_i})$ and
$\omega^{k_i}\rightarrow \mathbf{0}$. Proposition \ref{sublimit} indicates that $\mathbf{0}\in \partial \mathcal{L}_{\alpha}(d^*)$, i.e. $d^*\in \textrm{cri} (\mathcal{L}_{\alpha})$.

(2) This item follows as a consequence of the definition of the limit point.

(3) The continuity of $\mathcal{L}_{\alpha}(d)$ directly yields this result.
\end{proof}

\begin{theorem}[Convergence result]\label{confinal}
Suppose that $f$ and $g$ are all closed proper semi-algebraic functions.
Assume that
conditions (\ref{conditionp}), (\ref{conpara}), and (\ref{quad}) and one of \textbf{B1}, \textbf{B2}, \textbf{B3} hold. Let the sequence $\{(x^k,y^k,p^k)\}_{k=1,2,3,\ldots}$
 generated by ILR-ADMM be bounded. Then, the sequence $\{z^k=(x^k,y^k)\}_{k=0,1,2,3,\ldots}$ has finite length, i.e.
\begin{equation}
\sum_{k=0}^{+\infty}\|z^{k+1}-z^k\|_2<+\infty.
\end{equation}
And $\{(x^k,y^k,p^k)\}_{k=1,2,3,\ldots}$ converges to some $(x^*,y^*,p^*)$, which is a critical point  of $\mathcal{L}_{\alpha}(x,y,p)$.
\end{theorem}
\begin{proof}
Obviously, $\mathcal{L}_{\alpha}$ is also semi-algebraic. And with Lemma \ref{semikl}, $\mathcal{L}_{\alpha}$ is KL. From Lemma \ref{station}, $\mathcal{L}_{\alpha}$ is constant on $\mathcal{M}(d^0)$. Let $d^*$ be a stationary point of $\{d^k\}_{k=0,1,2,\ldots}$. Also from Lemma \ref{station}, we have $\textrm{dist}(d^k,\mathcal{M}(d^0))<\varepsilon$ and $ \mathcal{L}_{\alpha}(d^k)< \mathcal{L}_{\alpha}(d^*)+\eta$ if any $k>K$ for some $K$. If for some $k'>K$, $\textrm{dist}(\textbf{0},\partial \mathcal{L}_{\alpha}(d^{k'}))=0$, that is $d^{k'}\in \mathcal{M}(d^0)$. With Lemma \ref{station}, $ \mathcal{L}_{\alpha}(d^{k})= \mathcal{L}_{\alpha}(d^{k'})$ when $k\geq k'$. Thus, with Lemma \ref{descend}, $z^k=z^{k'}$ when $k\geq k'$.
If $\textrm{dist}(\textbf{0},\partial \mathcal{L}_{\alpha}(d^k))\neq 0$ for any $k>K$, with Lemma \ref{con}, we have
\begin{equation}
    \textrm{dist}(\textbf{0},\partial \mathcal{L}_{\alpha}(d^k))\cdot\varphi'(\mathcal{L}_{\alpha}(d^k)-\mathcal{L}_{\alpha}(d^*))\geq 1,
\end{equation}
which together with Lemma \ref{relative} gives
\begin{align}
    &\frac{1}{\varphi'(\mathcal{L}_{\alpha}(d^k)-\mathcal{L}_{\alpha}(d^*))}\leq\textrm{dist}(\textbf{0},\partial \mathcal{L}_{\alpha}(d^k))\leq \tau\|z^{k+1}-z^{k}\|_2.
\end{align}
Then, the concavity of $\varphi$ yields
\begin{align}
    &\mathcal{L}_{\alpha}(d^k)-\mathcal{L}_{\alpha}(d^{k+1})\nonumber\\
    &\quad\quad=\mathcal{L}_{\alpha}(d^k)-\mathcal{L}_{\alpha}(d^*)
    -[\mathcal{L}_{\alpha}(d^{k+1})-\mathcal{L}_{\alpha}(d^*)]\nonumber\\
    &\quad\quad\leq\frac{\varphi[\mathcal{L}_{\alpha}(d^k)-\mathcal{L}_{\alpha}(d^*)]-\varphi[\mathcal{L}_{\alpha}(d^{k+1})-\mathcal{L}_{\alpha}(d^*)]}
    {\varphi'[\mathcal{L}_{\alpha}(d^k)-\mathcal{L}_{\alpha}(d^*)]}\nonumber\\
    &\quad\quad\leq\{\varphi[\mathcal{L}_{\alpha}(d^k)-\mathcal{L}_{\alpha}(d^*)]-\varphi[\mathcal{L}_{\alpha}(d^{k+1})-\mathcal{L}_{\alpha}(d^*)]\}\times\tau\|z^{k+1}-z^{k}\|_2.\nonumber
\end{align}
With Lemma \ref{descend}, we have
\begin{align}
   &\nu\|z^{k+1}-z^k\|_2^2\leq\{\varphi[\mathcal{L}_{\alpha}(d^k)-\mathcal{L}_{\alpha}(d^*)]-\varphi[\mathcal{L}_{\alpha}(d^{k+1})-\mathcal{L}_{\alpha}(d^*)]\}\times\tau \|z^{k+1}-z^{k}\|_2,\nonumber
\end{align}
which is equivalent to
\begin{align}
   &2\frac{\nu}{\tau}\|z^{k+1}-z^k\|_2\leq2\times\sqrt{\varphi[\mathcal{L}_{\alpha}(d^k)-\mathcal{L}_{\alpha}(d^*)]-\varphi[\mathcal{L}_{\alpha}(d^{k+1})-\mathcal{L}_{\alpha}(d^*)]}\times\sqrt{\frac{\nu}{\tau}}\sqrt{\|z^{k+1}-z^{k}\|_2}.
\end{align}
Using the Schwartz's inequality, we then derive that
\begin{align}
    &2\frac{\nu}{\tau}\|z^{k+1}-z^k\|_2\leq\{\varphi[\mathcal{L}_{\alpha}(d^k)-\mathcal{L}_{\alpha}(d^*)]-\varphi[\mathcal{L}_{\alpha}(d^{k+1})-\mathcal{L}_{\alpha}(d^*)]\}+\frac{\nu}{\tau}\|z^{k+1}-z^{k}\|_2.
\end{align}
That is also
\begin{align}\label{core}
    &\frac{\nu}{\tau}\|z^{k+1}-z^k\|_2\leq\varphi[\mathcal{L}_{\alpha}(d^k)-\mathcal{L}_{\alpha}(d^*)]-\varphi[\mathcal{L}_{\alpha}(d^{k+1})-\mathcal{L}_{\alpha}(d^*)].
\end{align}
Summing (\ref{core}) from $K$ to $K+j$ yields that
\begin{align}
   &\frac{\nu}{\tau} \sum_{k=K}^{K+j}\|z^{k+1}-z^k\|_2\leq \varphi[\mathcal{L}_{\alpha}(d^K)-\mathcal{L}_{\alpha}(d^*)]- \varphi[\mathcal{L}_{\alpha}(d^{K+j+1})-\mathcal{L}_{\alpha}(d^*)]\leq \varphi[\mathcal{L}_{\alpha}(d^K)-\mathcal{L}_{\alpha}(d^*)].
\end{align}
Letting $j\rightarrow+\infty$,  we have
\begin{equation}
   \frac{\nu}{\tau} \sum_{k=K}^{+\infty}\|z^{k+1}-z^k\|_2\leq\varphi[\mathcal{L}_{\alpha}(d^K)-\mathcal{L}_{\alpha}(d^*)]<+\infty.
\end{equation}
From Lemma \ref{station}, there exists a critical  point $(x^*,y^*,p^*)$ of $\mathcal{L}_{\alpha}(x,y,p)$. Then, $\{z^k\}_{k=0,1,2,\ldots}$ is convergent and $(x^*,y^*)$ is a stationary point of $\{z^k\}_{k=0,1,2,\ldots}$. That is to say  $\{z^k\}_{k=0,1,2,\ldots}$ converges to $(x^*,y^*)$. Note that $p^{k}$ is a linear composition of $x^k$  and $y^k$, thus $\{p^k\}_{k=0,1,2,\ldots}$ converges to $p^*$.
\end{proof}

\section{Applications and numerical results}
In this part, we consider   applying ILR-ADMM to problem (\ref{tvq}) for  image deblurring. This section contains two parts: in the first one, several
basic properties of the proposed algorithm, such as the convergence  and
the influence of the selection of the  parameter $q$ in problem (\ref{tvq}), are investigated;
in the second one, the proposed algorithm is compared with other classical
methods for image deblurring. We employ four images (see Figure 1), which include three nature
images, one MRI image  for our numerical experiments. The
performance of the deblurring algorithms is quantitatively measured by means of  the signal-to-noise ratio (SNR)
\begin{align}
\textrm{SNR}(u,u^*)=10*\log_{10}(\frac{\|u-\bar{u}\|^2}{\|u-u^*\|^2}),
\end{align}
where  $u$ and $u^*$ denote the original image and the restored image, respectively, and $\bar{u}$ represents the mean of the original image $u$.
\begin{figure}[htbp]
\centering
   \subfigure[]{
    \includegraphics[width=2.5in]{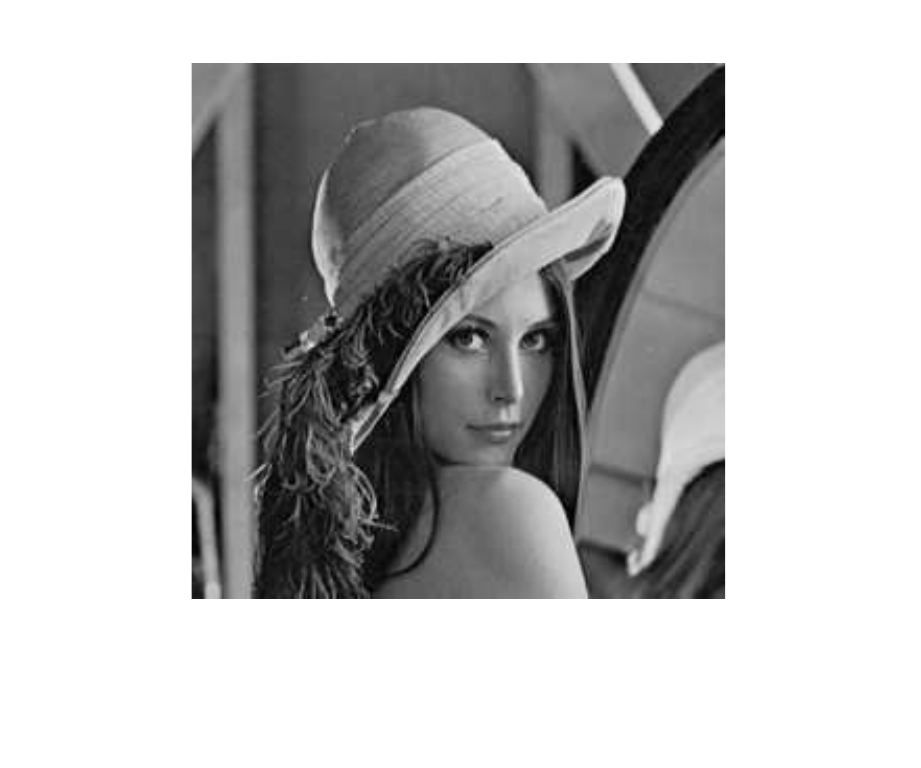}}
  \subfigure[]{
    \includegraphics[width=2.5in]{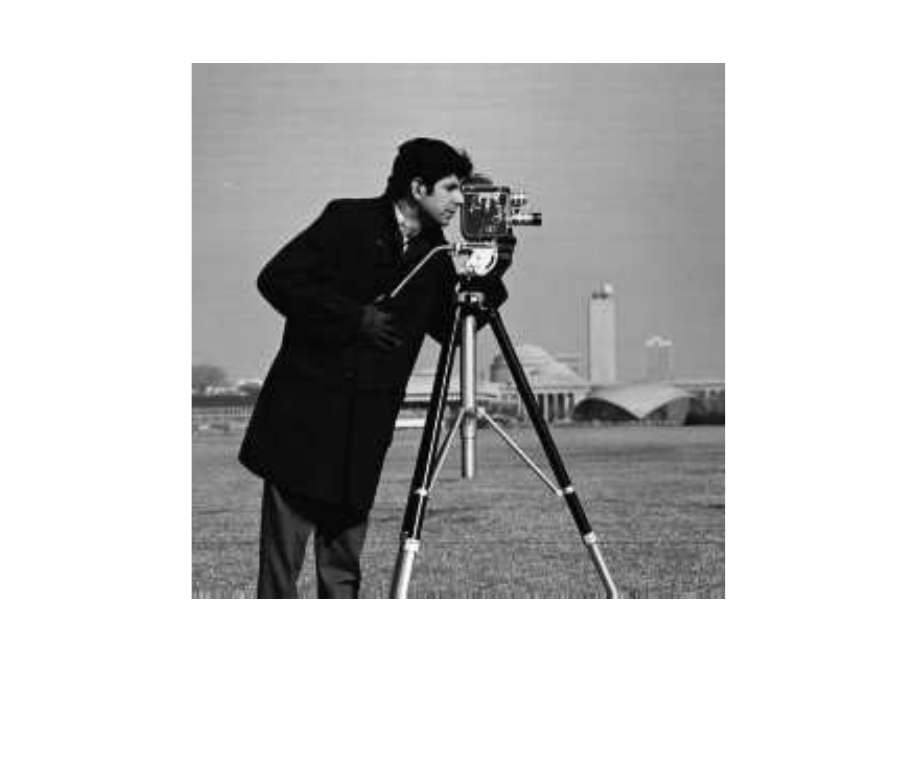}}
  \subfigure[]{
    \includegraphics[width=2.5in]{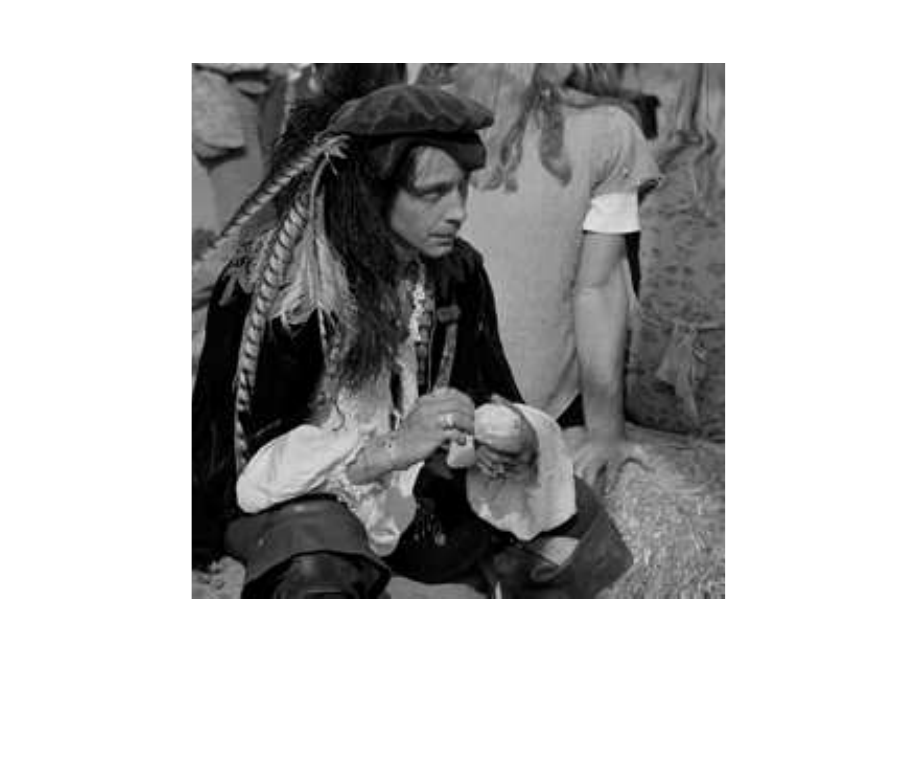}}
  \subfigure[]{
    \includegraphics[width=2.5in]{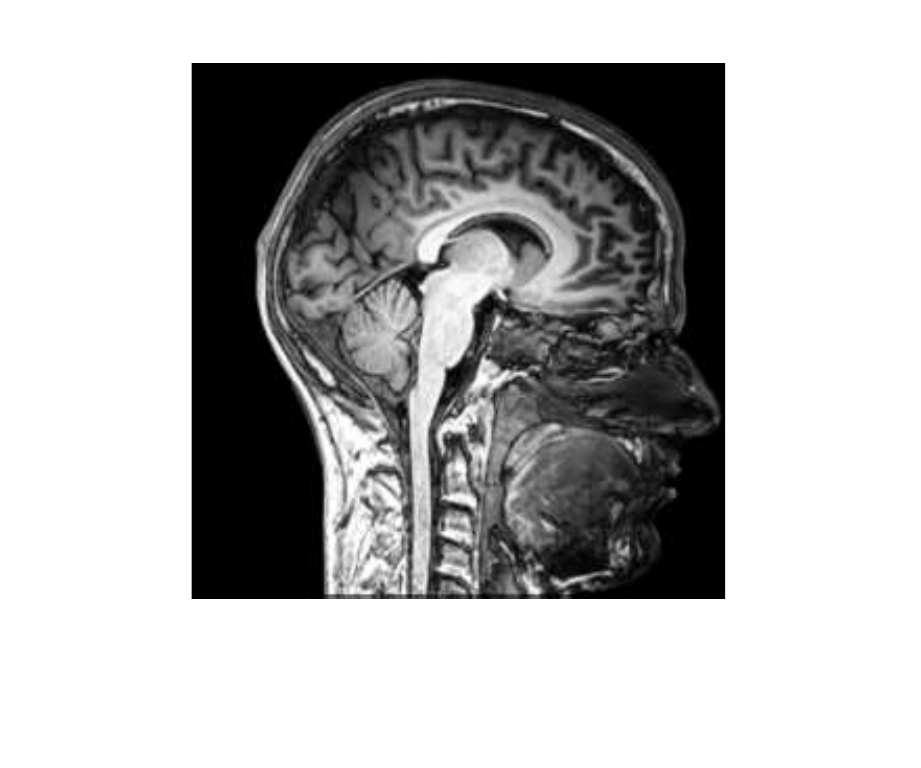}}
  \caption{Original images. (a) Lena ($256\times256$); (b) Cameraman ($256\times256$); (c) Orangeman ($256\times256$);
(d) Brain vessels ($256\times256$).}
\end{figure}
In the experiments, we use an increasing technique for the parameter $\alpha$ in ILR-ADMM, i.e., $\alpha=\min\{\rho\alpha, \alpha_{\max}\}$, with  $\rho=1.05$ and upper bound $\alpha_{\max}=10^3$. Such a technique has been used for ADMM in \cite{lu2017nonconvex}. From (\ref{conpara}), we need $r>\alpha$; and it is set as $r=\alpha+10^{-6}$. For all the algorithms used in this section, the initialization is the blurred image.
\subsection{Performance of ILR-ADMM}
In this subsection, we focus on the convergence of ILR-ADMM. The blurring operator used for our experiments are generated
by the matlab command \verb"fspecial('gaussian',17,5)". And in the problem (\ref{tvq}), we choose $\varepsilon=10^{-7}$.
 For $q=0.2,0.4,0.6,0.8$, we apply  ILR-ADMM to reconstruct the four images in Figure 1. We run all the algorithms  10 times in each case and then take the average.  Fig. 2 shows the SNR versus the iterations and the maximum iteration is 200.
\begin{figure}[htbp]
\centering
   \subfigure[]{
    \includegraphics[width=2.5in]{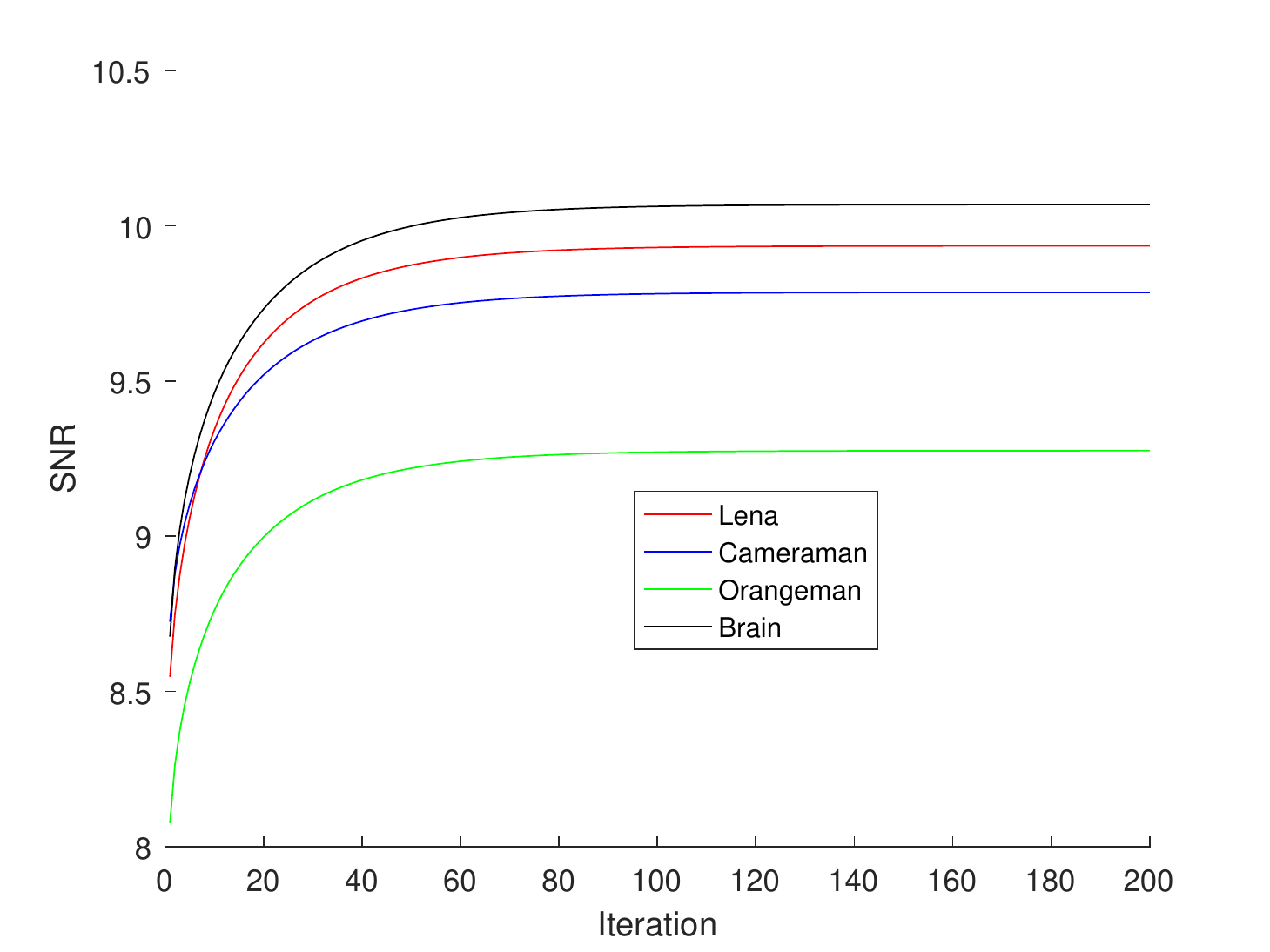}}
  \subfigure[]{
    \includegraphics[width=2.5in]{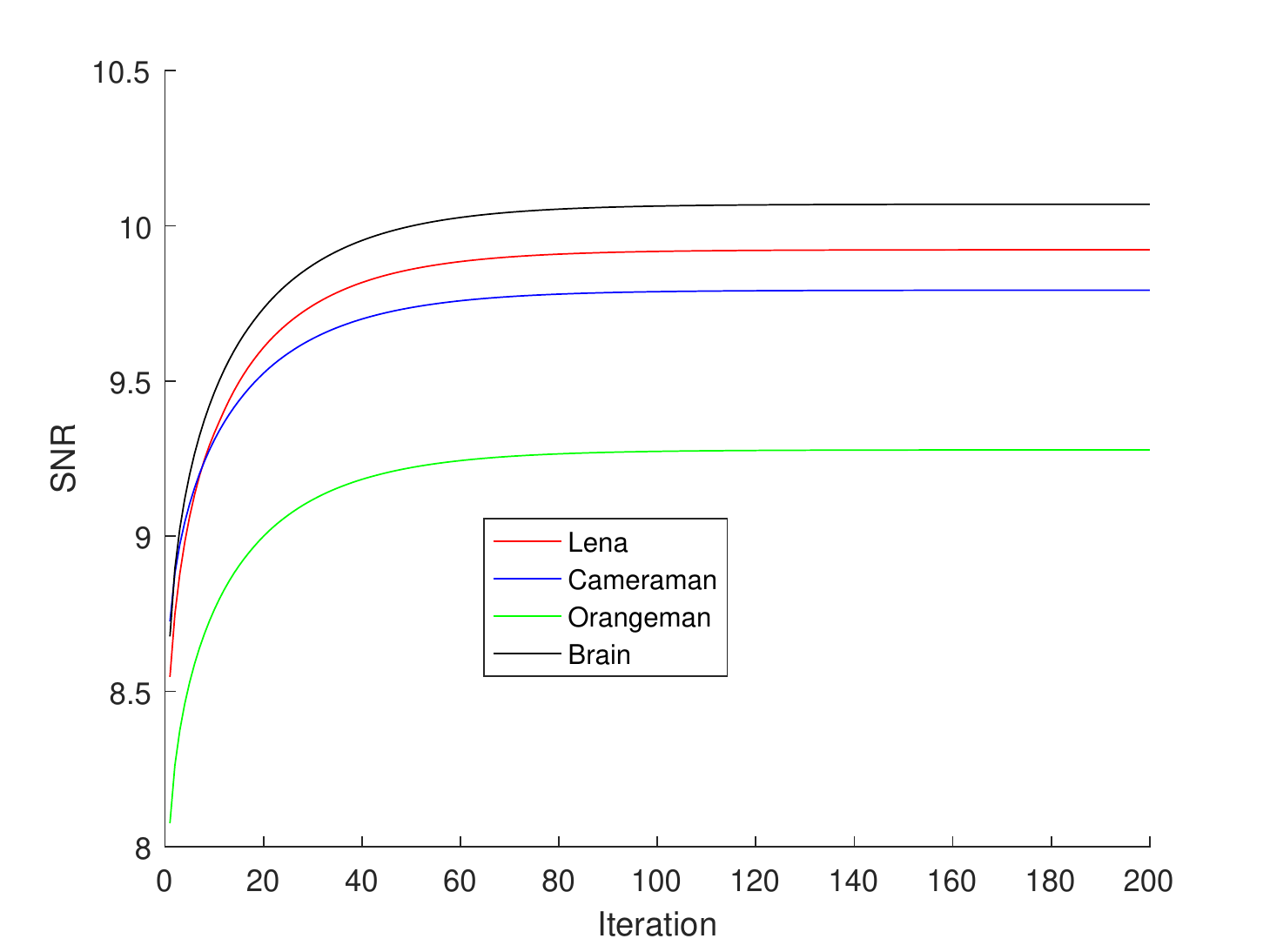}}
  \subfigure[]{
    \includegraphics[width=2.5in]{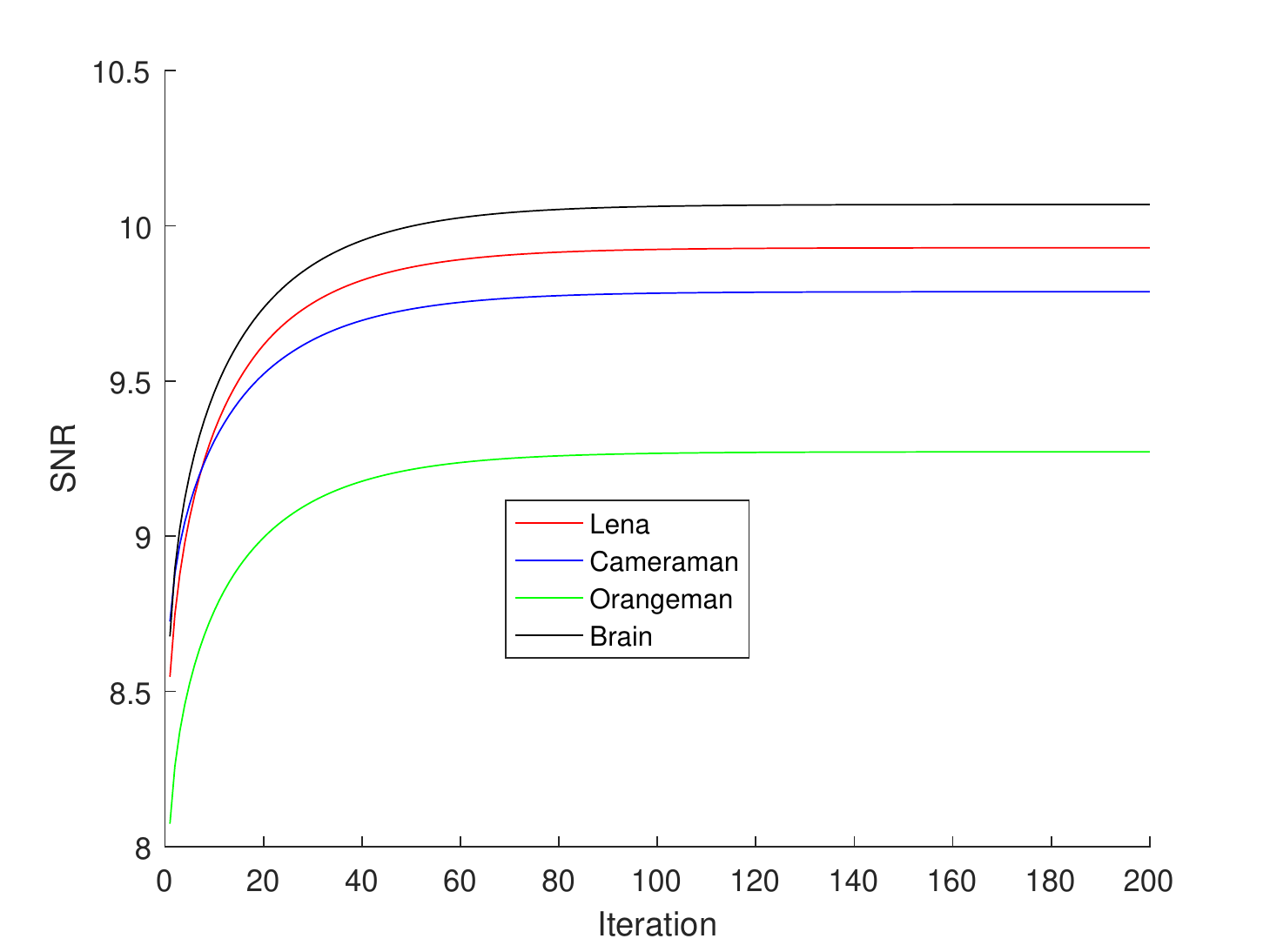}}
  \subfigure[]{
    \includegraphics[width=2.5in]{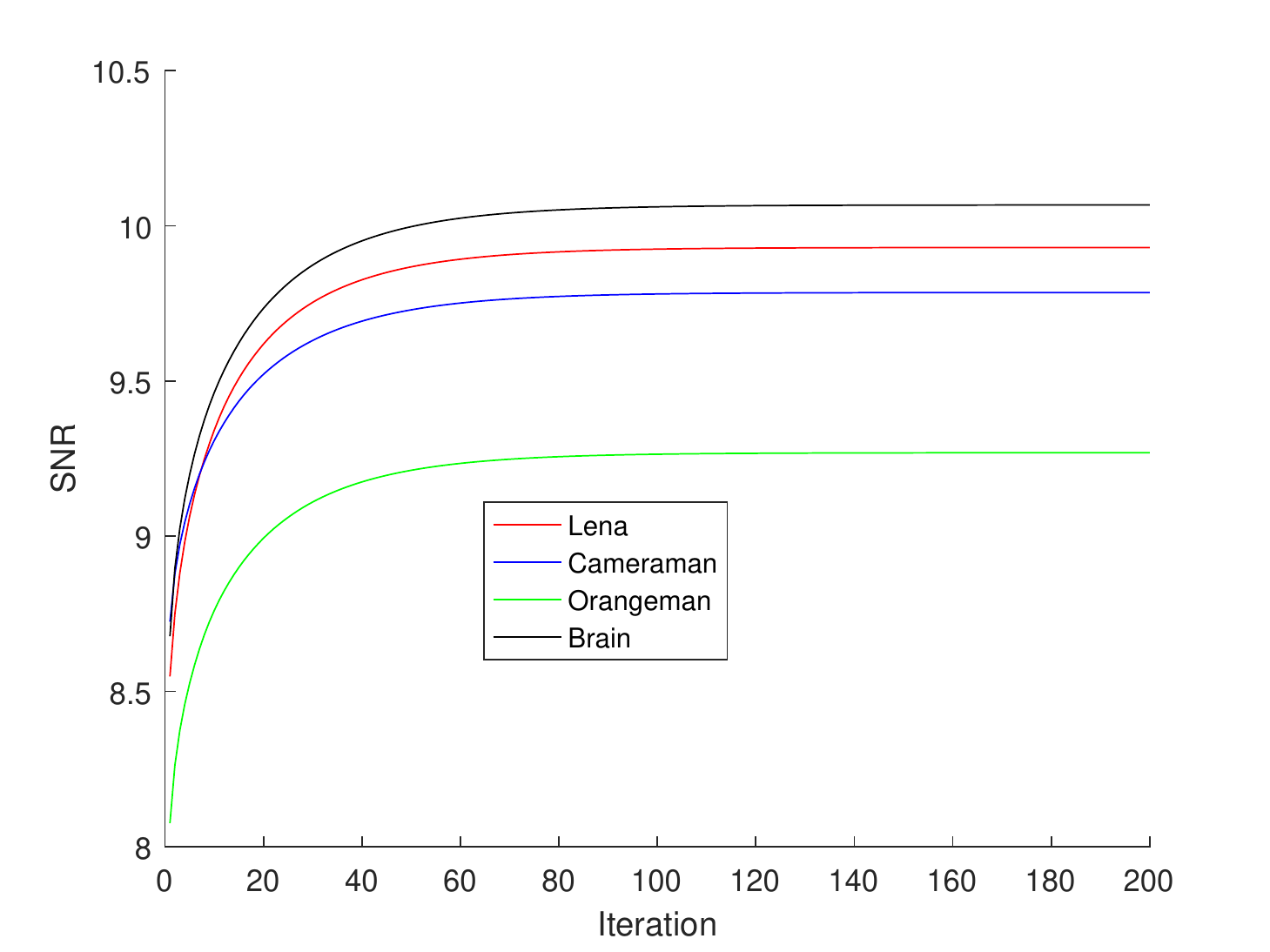}}
  \caption{SNR versus the iterations for parameters. (a) $q=0.2$; (b)$q=0.4$; (c) $q=0.6$;
(d) $q=0.8$.}
\end{figure}
\subsection{Comparisons with other classical methods}
This subsection focuses on $q=\frac{1}{2}$ for problem (\ref{tvq}). We consider two algorithms for comparisons: the first one is the direct nonconvex ADMM; the second one considers  using an inner loop  for the subproblem. Precisely, the inner loop is constructed by the proximal reweighted algorithm. And in the numerical examples, the inner loop is set as 10. We call this algorithm as in-loop-ADMM.

The blurring operator $\Psi$ is generated  by the Matlab commands  \verb"fspecial('gaussian',.,.)". The parameter is set as $\sigma=10^{-4}$, $e$ is generated  by the Gaussian noise $\mathcal{N}(0,0.01)$.  Fig. 3, 4, 5 and 6 present the reconstructed images  with different algorithms for the four images in Fig. 1. The  maximum iteration is set as 200. We run all algorithms ten times and take the average. The time cost (T) is also reported. We also plot the SNR versus the iteration  for different algorithms in each image. From numerical results, we can see ILR-ADMM perform better than the nonconvex ADMM with almost same time cost. Due to that in-loop-ADMM employs the inner loop, the time cost is much larger than   ILR-ADMM. In general, ILR-ADMM  outperforms than the other two algorithms.
\begin{figure}
  \centering
  \subfigure[Blurred and noised image]{
    \includegraphics[width=2.5in]{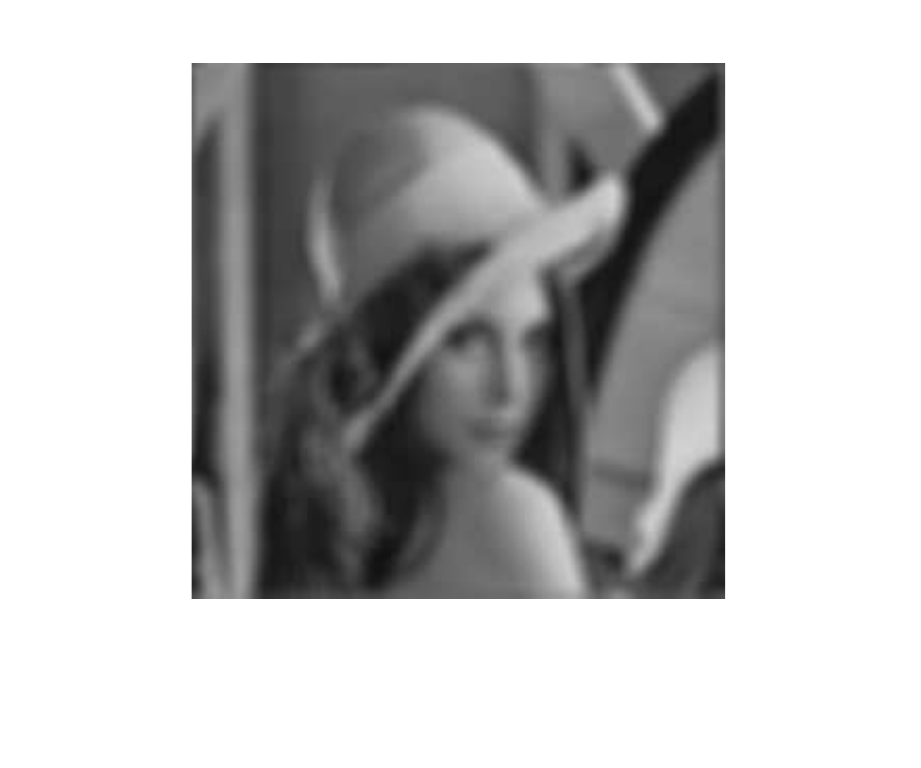}}
  \subfigure[Recovery by ILR-ADMM, SNR=11.73dB, T=2.1s]{
    \includegraphics[width=2.5in]{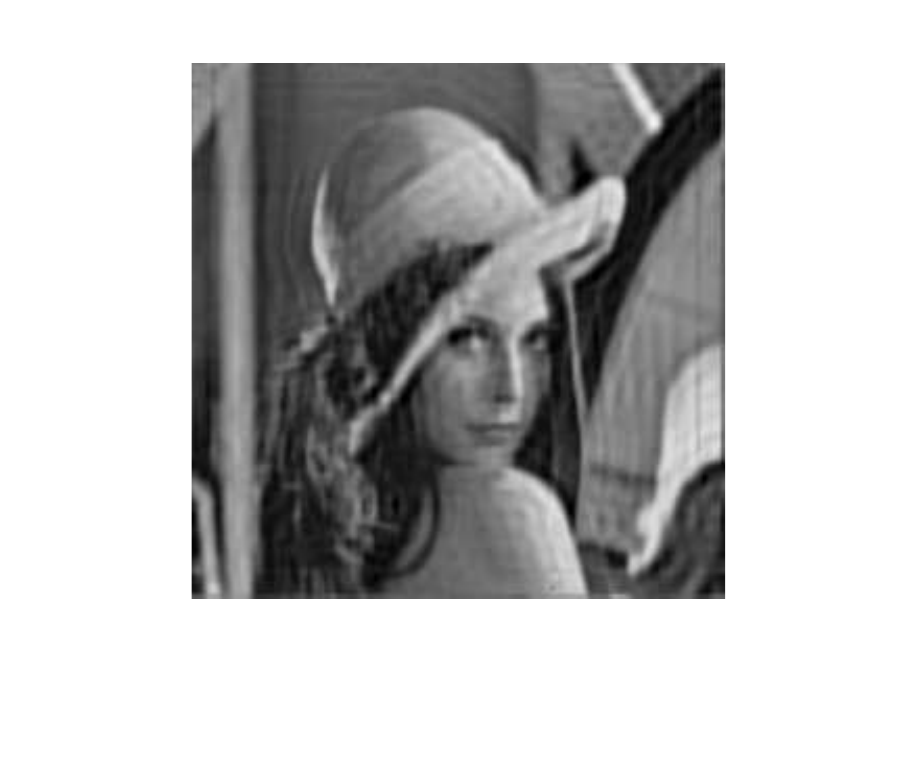}}
  \subfigure[Recovery by nonconvex ADMM, SNR=11.65dB, T=2.3s]{
    \includegraphics[width=2.5in]{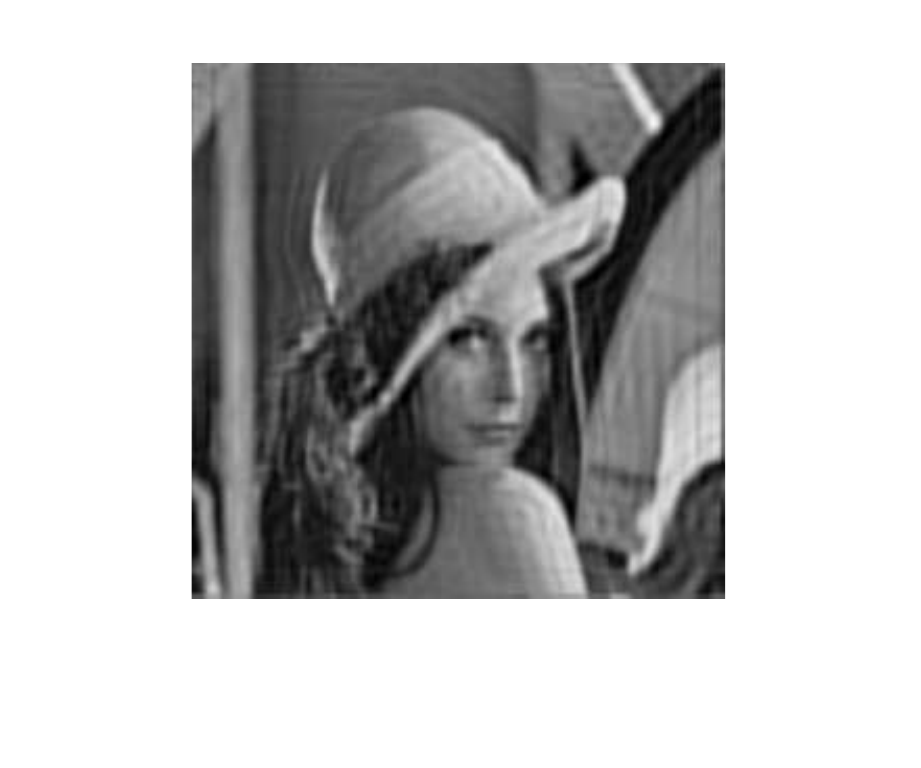}}
      \subfigure[Recovery by in-loop-ADMM, SNR=11.35dB, T=25.2s ]{
    \includegraphics[width=2.5in]{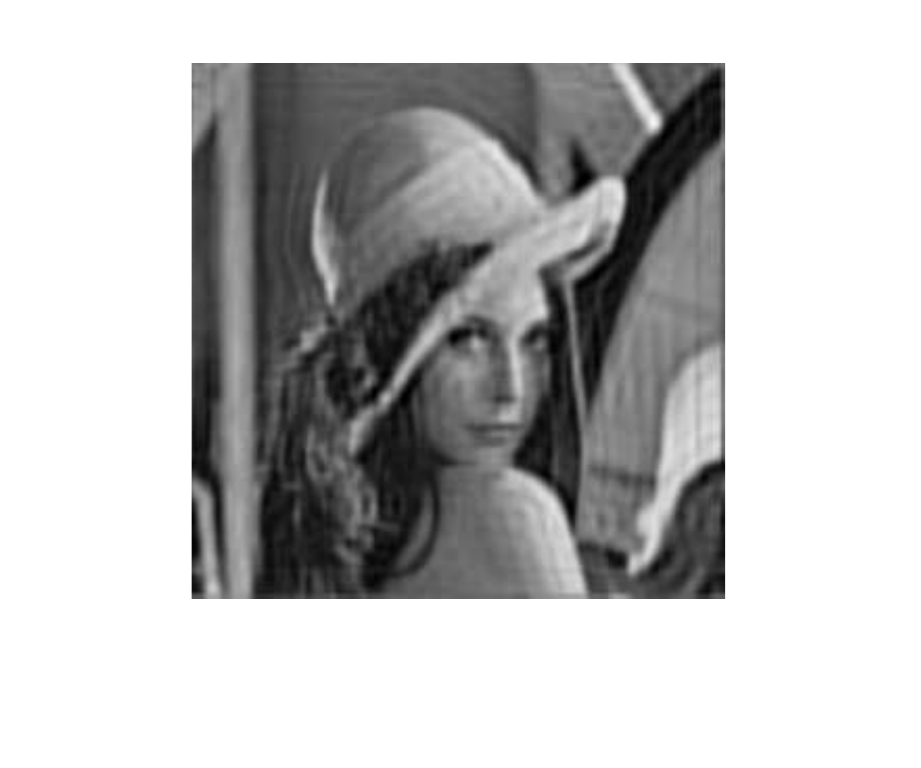}}
      \subfigure[SNR versus the iteration  for different algorithms]{
    \includegraphics[width=2.5in]{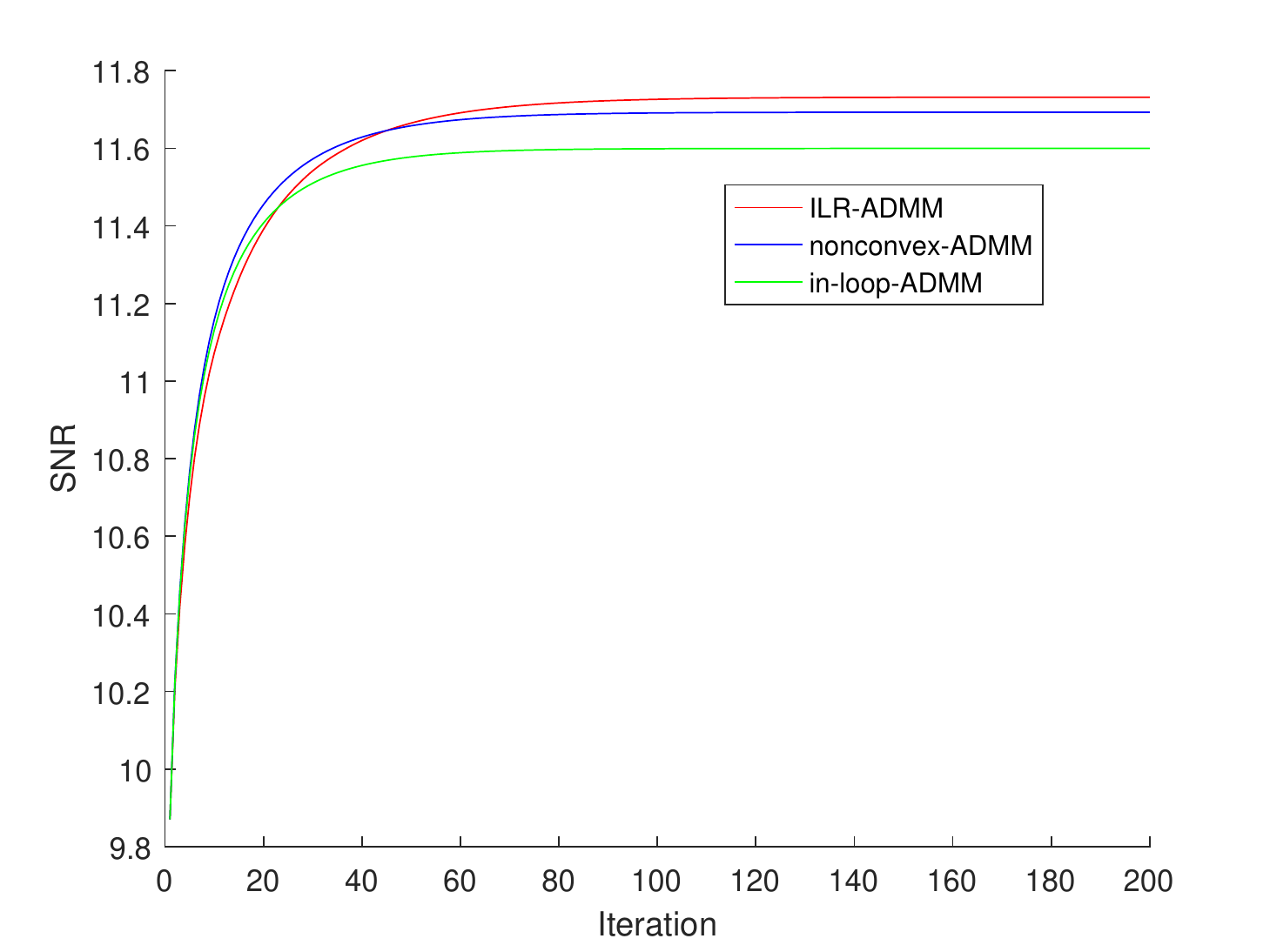}}
\caption{Reconstructed  images  by different methods for ``Lena" image}
\end{figure}

\begin{figure}
  \centering
  \subfigure[Blurred and noised image]{
    \includegraphics[width=2.5in]{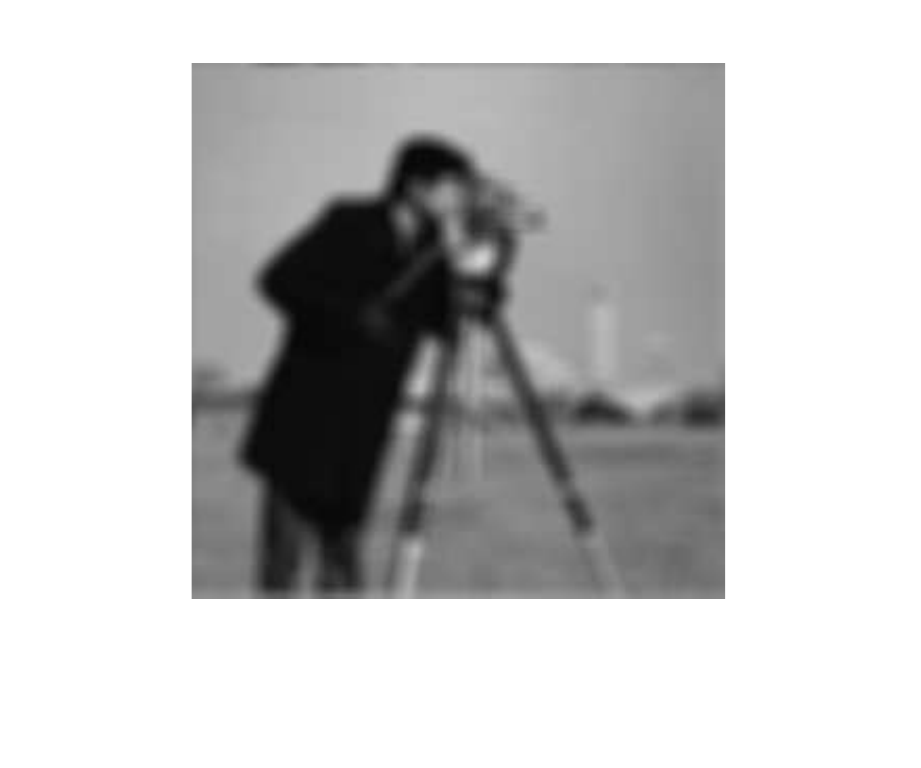}}
  \subfigure[Recovery by ILR-ADMM, SNR=11.53dB, T=2.0s]{
    \includegraphics[width=2.5in]{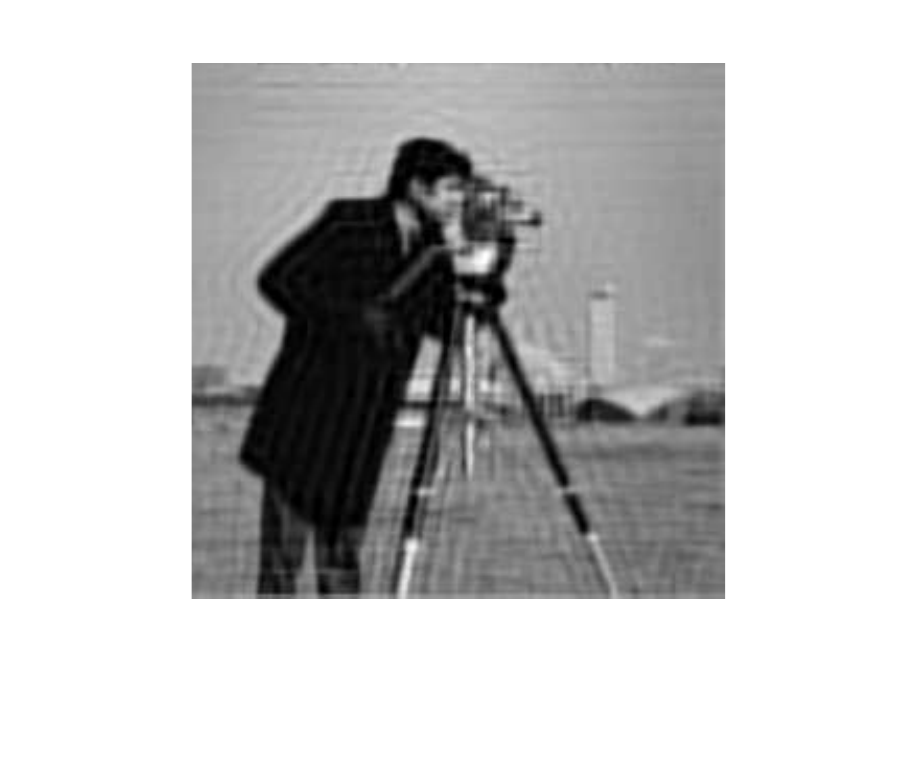}}
  \subfigure[Recovery by nonconvex ADMM, SNR=11.45dB, T=1.9s]{
    \includegraphics[width=2.5in]{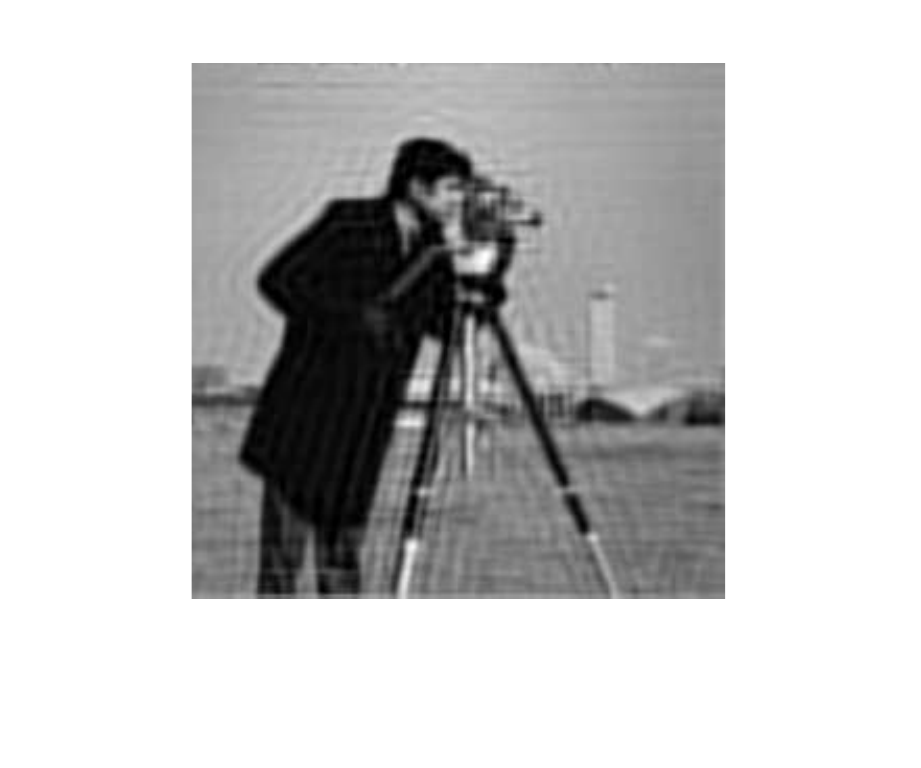}}
      \subfigure[Recovery by in-loop-ADMM, SNR=11.39dB, T=22.3s ]{
    \includegraphics[width=2.5in]{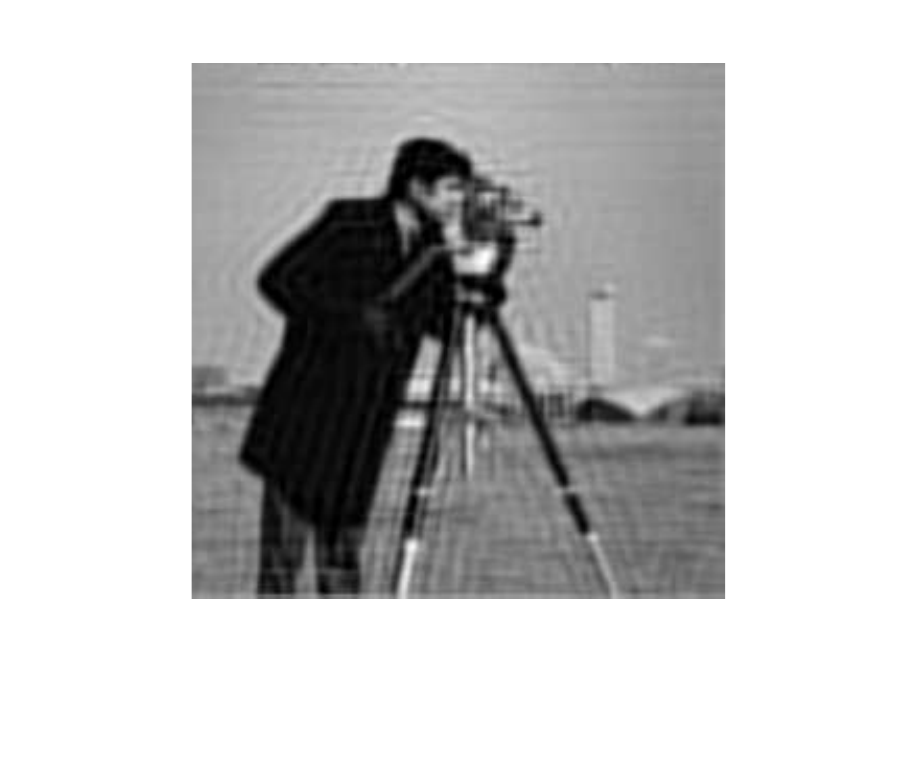}}
      \subfigure[SNR versus the iteration  for different algorithms]{
    \includegraphics[width=2.5in]{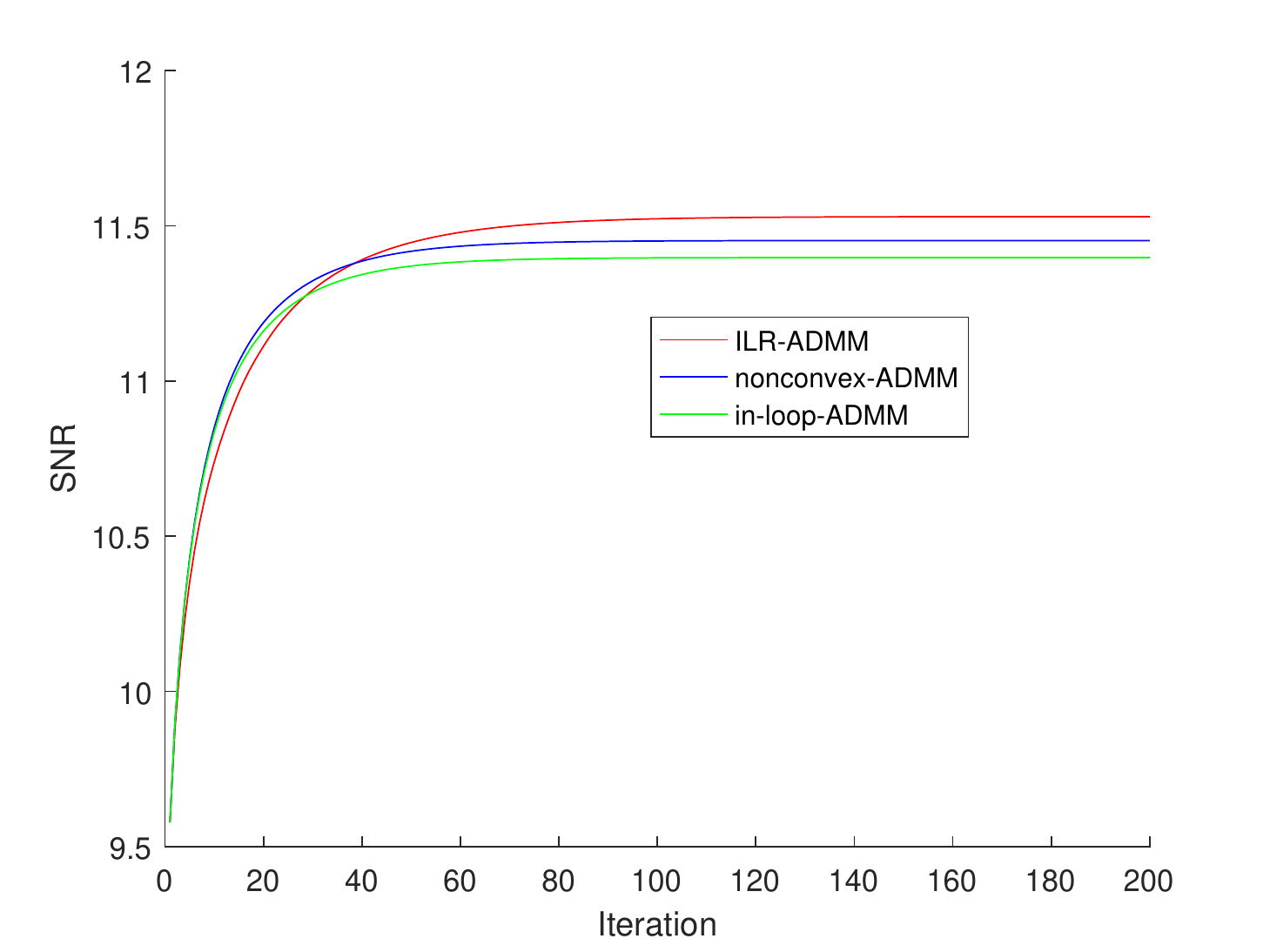}}
\caption{Reconstructed  images  by different methods for ``Cameraman" image}
\end{figure}

\begin{figure}
  \centering
  \subfigure[Blurred and noised image]{
    \includegraphics[width=2.5in]{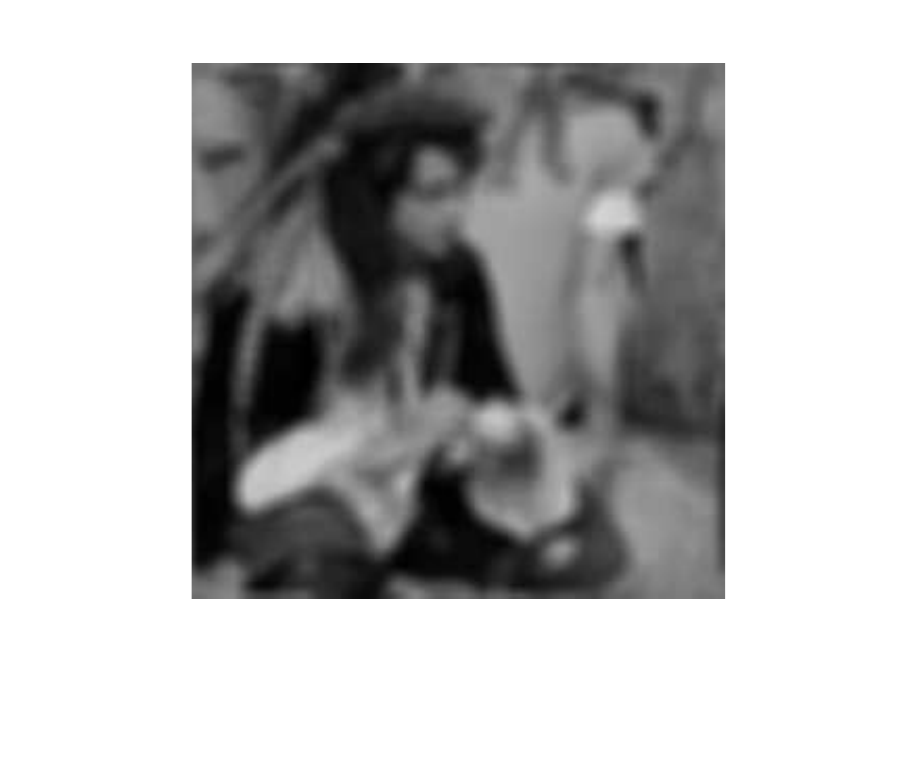}}
  \subfigure[Recovery by ILR-ADMM, SNR=11.25dB, T=2.1s]{
    \includegraphics[width=2.5in]{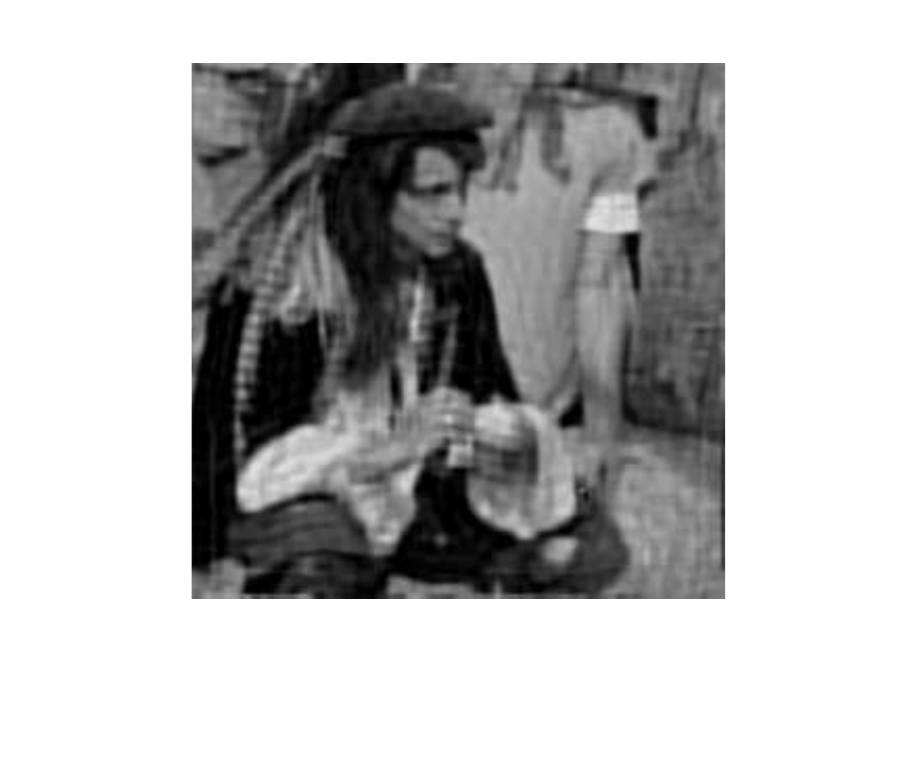}}
  \subfigure[Recovery by nonconvex ADMM, SNR=11.13dB, T=2.2s]{
    \includegraphics[width=2.5in]{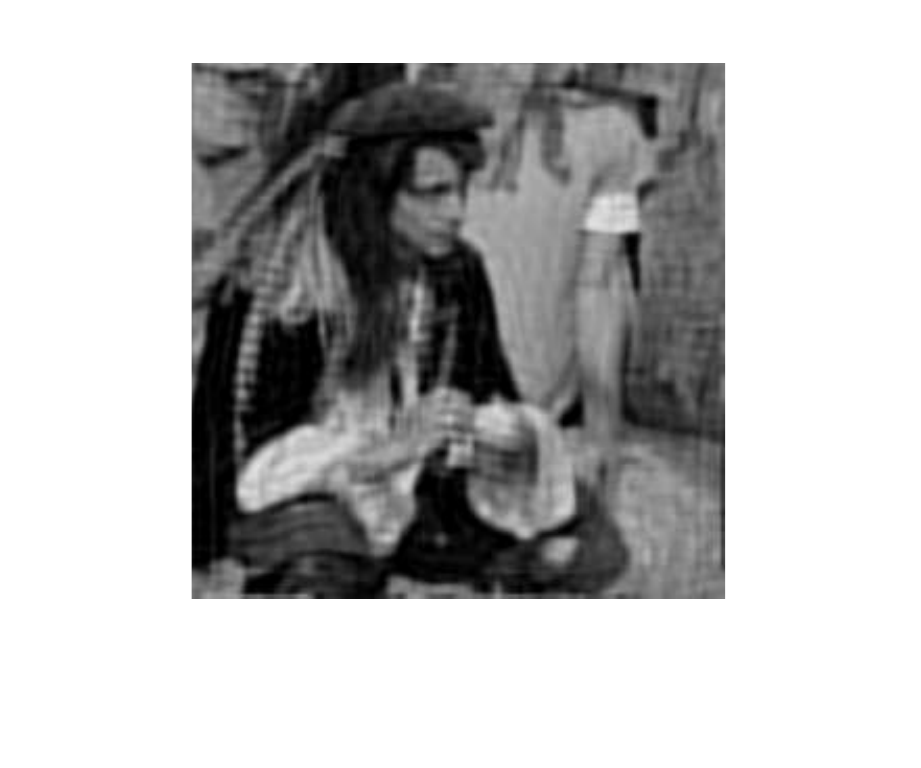}}
      \subfigure[Recovery by in-loop-ADMM, SNR=11.07dB, T=21.4s ]{
    \includegraphics[width=2.5in]{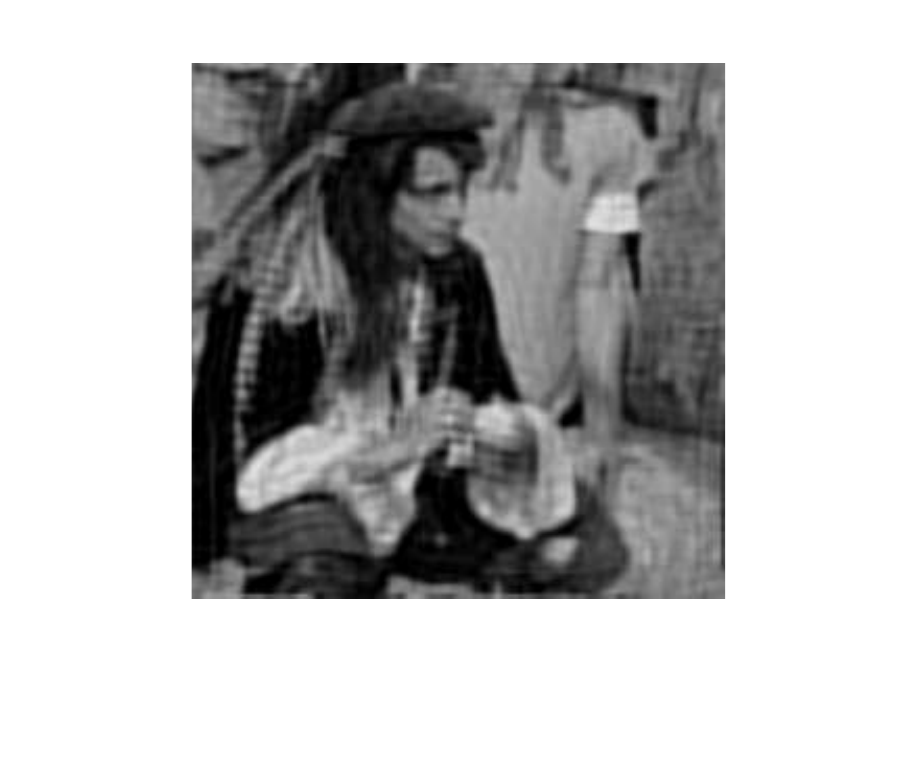}}
      \subfigure[SNR versus the iteration  for different algorithms]{
    \includegraphics[width=2.5in]{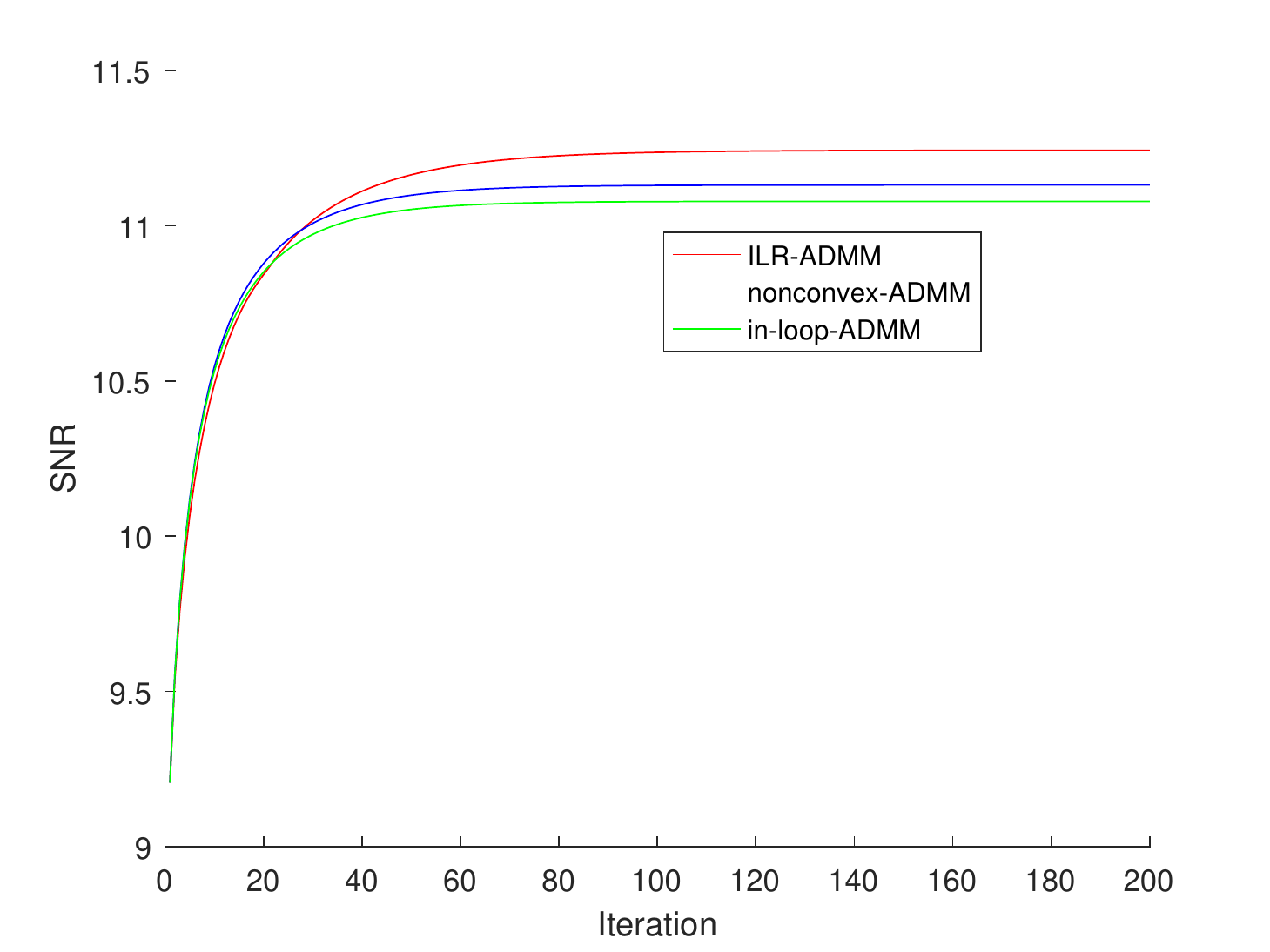}}
\caption{Reconstructed  images  by different methods for ``Cameraman" image}
\end{figure}

\begin{figure}
  \centering
  \subfigure[Blurred and noised image]{
    \includegraphics[width=2.5in]{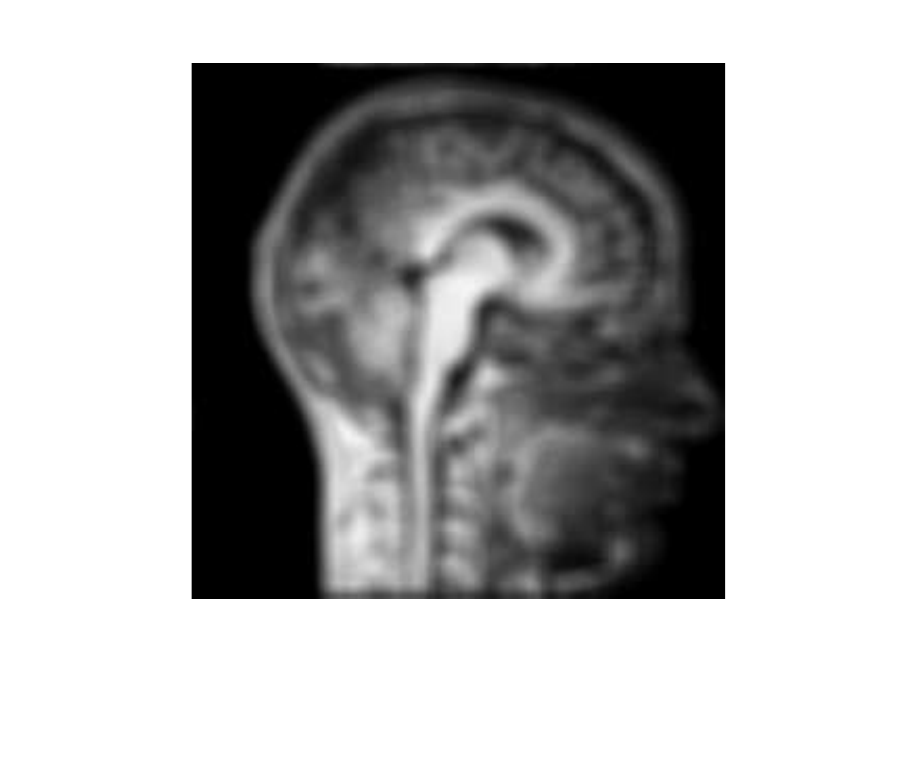}}
  \subfigure[Recovery by ILR-ADMM, SNR=12.53dB, T=2.0s]{
    \includegraphics[width=2.5in]{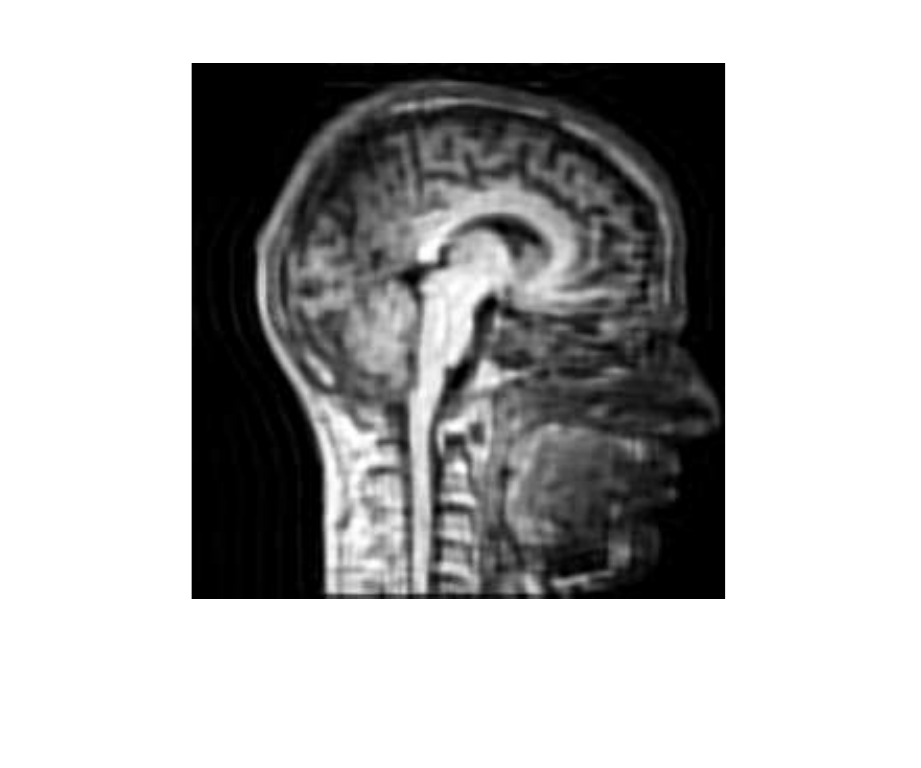}}
  \subfigure[Recovery by nonconvex ADMM, SNR=12.39dB, T=2.2s]{
    \includegraphics[width=2.5in]{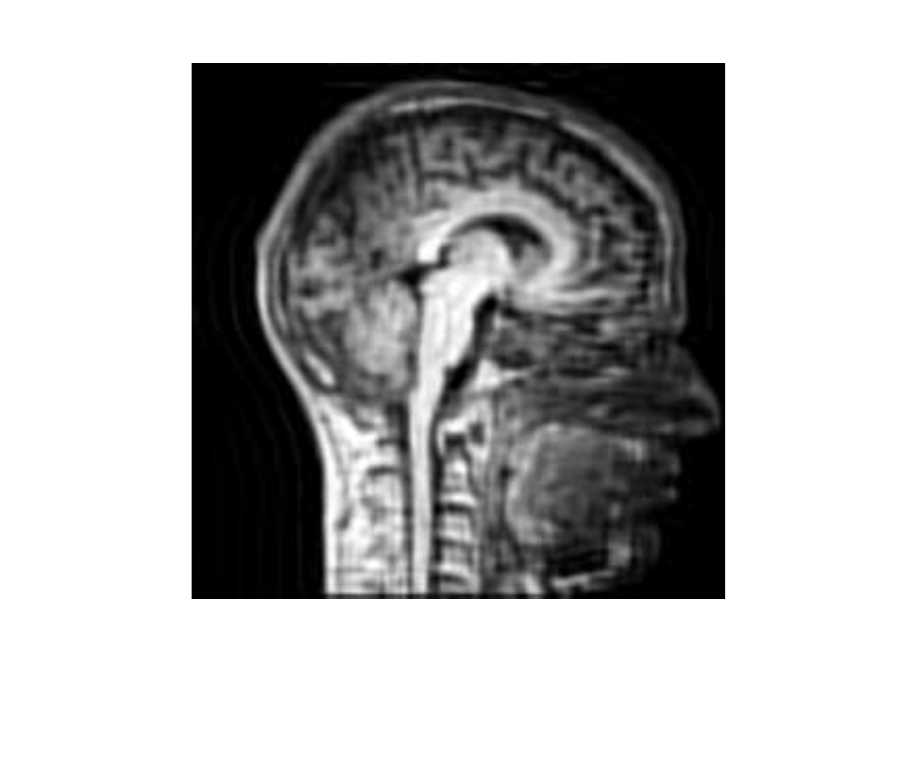}}
      \subfigure[Recovery by in-loop-ADMM, SNR=12.35dB, T=23.7s]{
    \includegraphics[width=2.5in]{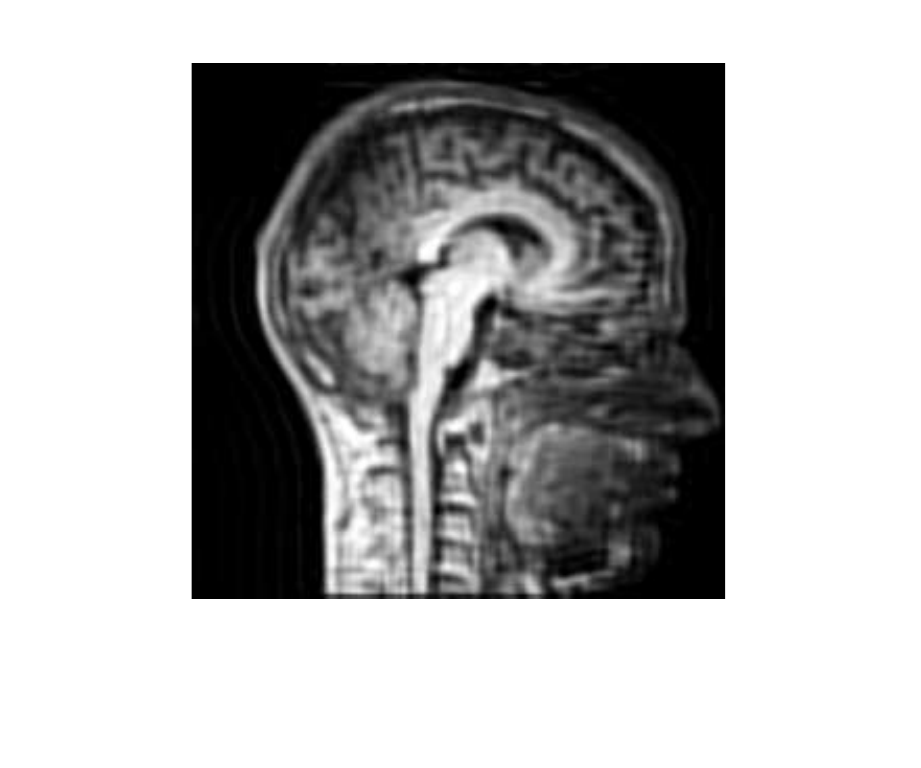}}
      \subfigure[SNR versus the iteration  for different algorithms]{
    \includegraphics[width=2.5in]{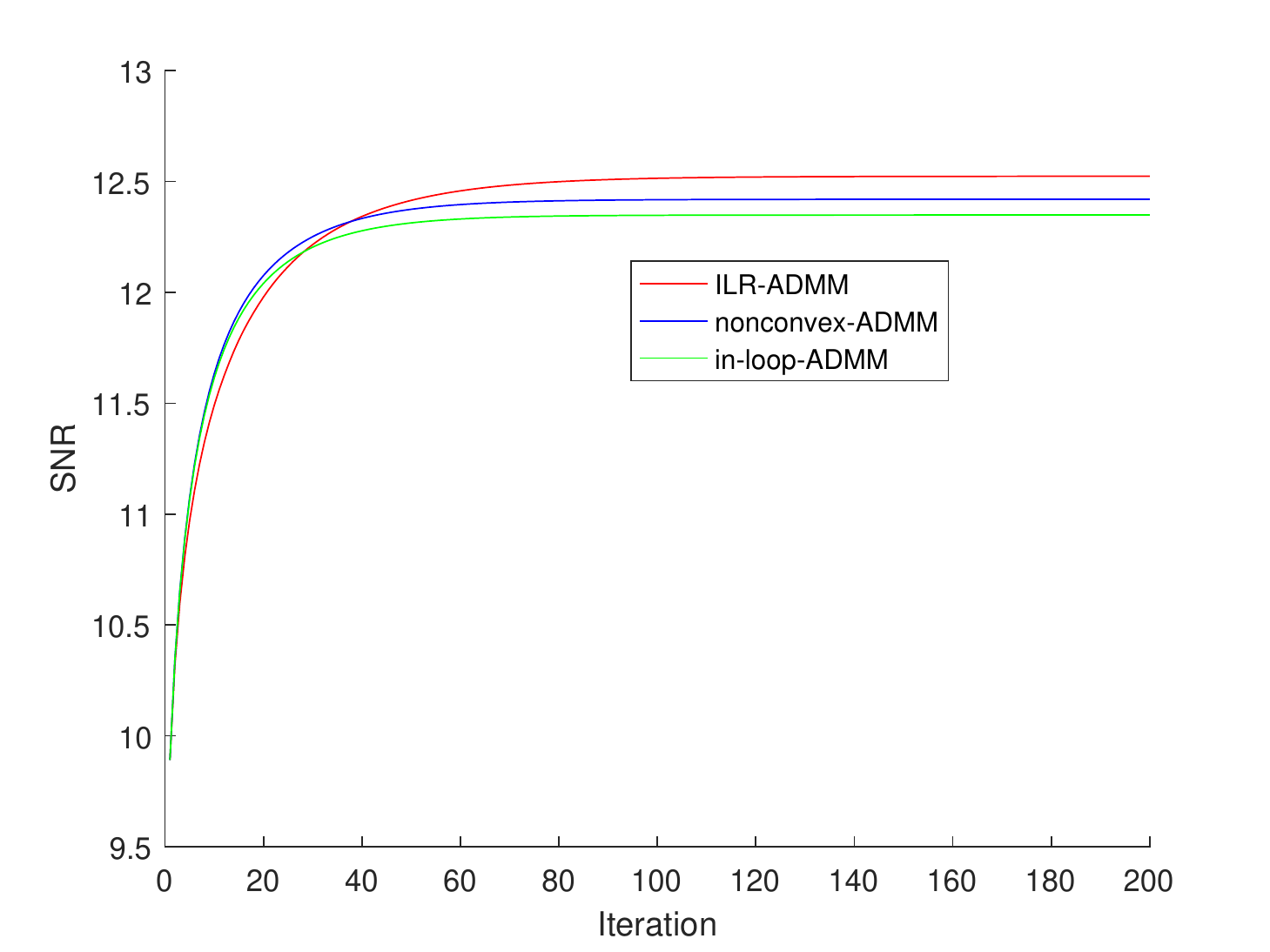}}
\caption{Reconstructed  images  by different methods for ``Brain" image}
\end{figure}
\section{Conclusion}
In this paper, we consider a class of nonconvex and nonsmooth minimizations with linear constraints which have applications in signal processing and machine learning research. The classical ADMM method for these problems always encounters both computational and mathematical barriers in solving the subproblem. We   combined the reweighted algorithm and linearized techniques, and then designed a new ADMM.
In the proposed algorithm, each subproblem just needs to calculate the proximal maps. The convergence is proved under several assumptions on the parameters and functions. And numerical results demonstrate the efficiency of our algorithm.

\end{document}